\documentclass[11pt]{amsart}
\usepackage[margin = 1in]{geometry}
\usepackage{mathptmx}
\usepackage{latexsym}
\usepackage{amsmath}
\usepackage{amssymb, mathtools}
\usepackage{amsfonts} 
\usepackage{eucal} 
\usepackage{amsbsy}
\usepackage{amsthm}
\usepackage[all]{xy}
\usepackage[pagebackref=true]{hyperref}
\usepackage{graphicx}
\usepackage{xcolor}
\usepackage{enumerate}

\newtheorem{thm}{Theorem}[section]
\newtheorem*{thm*}{Theorem}
\newtheorem{lemma}[thm]{Lemma}
\newtheorem{prop}[thm]{Proposition}

\newtheorem*{cor*}{Corollary}

\theoremstyle{definition}
\newtheorem{defn}[thm]{Definition}

\theoremstyle{remark}
\newtheorem{remark}[thm]{Remark}

\numberwithin{equation}{section}

\newcommand {\real}  {\ensuremath{\mathbb{R}}}
\newcommand {\intg}  {\ensuremath{\mathbb{Z}}}

\newcommand {\cplx}  {\ensuremath{\mathbb{C}}}
\newcommand {\rat}   {\ensuremath{\mathbb{Q}}}

\newcommand {\rk}    {\operatorname{rk}}

\newcommand {\smlhf} {\ensuremath{\mbox{$\frac{1}{2}$}}}

\newcommand {\pro}   {\ensuremath{\operatorname{pr}}}

\newcommand {\pt}    {{\ensuremath{\operatorname{pt}}}}
\newcommand {\id}    {\ensuremath{\operatorname{id}}}

\newcommand {\BSO}   {\ensuremath{\operatorname{BSO}}}

\newcommand {\MSPL}   {\ensuremath{\operatorname{MSPL}}}

\newcommand {\SO}   {{\ensuremath{\operatorname{SO}}}}

\newcommand {\Witt}   {{\ensuremath{\operatorname{Witt}}}}

\newcommand {\MWITT}   {\ensuremath{\operatorname{MWITT}}}
\newcommand {\sign}   {{\ensuremath{\operatorname{sign}}}}

\newcommand {\Th}   {\ensuremath{\operatorname{Th}}}
\newcommand {\topo}   {\ensuremath{{\operatorname{top}}}}
\newcommand {\an}   {\ensuremath{{\operatorname{an}}}}

\newcommand {\geo}    {{\ensuremath{\operatorname{geo}}}}

\newcommand {\KO}   {{\ensuremath{\operatorname{KO}}}}
\newcommand {\K}   {{\ensuremath{\operatorname{K}}}}

\newcommand {\ko}   {{\ensuremath{\operatorname{ko}}}}
\newcommand {\ku}   {{\ensuremath{\operatorname{k}}}}
\newcommand {\KK}  {{\ensuremath{\operatorname{KK}}}}

\newcommand {\ism}   {\ensuremath{\intg [\smlhf]}}

\newcommand {\Spin}   {{\ensuremath{\operatorname{Spin}}}}

\def\Di{\mathfrak{D}\kern-6.5pt/}
\def\Spi{\mathfrak{S}\kern-6.5pt/}

\usepackage{mathrsfs}
\newcommand{\bb}[1]{\mathbb{#1}}

\newcommand\lra{\longrightarrow}
\newcommand\xlra[1]{\xrightarrow{\phantom{x} #1 \phantom{x}}}

\newcommand\eps\varepsilon
\newcommand\pa\partial

\DeclareMathAlphabet{\mathpzc}{OT1}{pzc}{m}{it}

\usepackage{xcolor}
\definecolor{darkgreen}{cmyk}{1,0,1,.2}
\definecolor{m}{rgb}{1,0.1,1}

\begin{document}


\title[K-orientations and Gysin maps]
  {Smooth atlas stratified spaces, K-Homology Orientations\\ and Gysin maps. Part 2.}

\author{Pierre Albin}

\address{Department of Mathematics, University of Illinois at Urbana-Champaign, USA}

\email{palbin@illinois.edu}

\author{Markus Banagl}

\address{Institut f\"ur Mathematik, Universit\"at Heidelberg,
  Im Neuenheimer Feld 205, 69120 Heidelberg, Germany}

\email{banagl@mathi.uni-heidelberg.de}

\author{Paolo Piazza}

\address{Dipartimento di Matematica, Sapienza Universit\`a di Roma, Italy}

\email{piazza@mat.uniroma1.it}

\date{May 26, 2026}





\begin{abstract}
In this Part 2 of our article we give a detailed discussion of the 
compatibility between the analytic Gysin maps we have defined  in Part 1 and the topological Gysin maps defined by the second author. 
A significant r\^ole is played by a bordism-like description of K-homology due to Jakob which is closely related to the geometric K-homology theory of Baum and Douglas. We give a self-contained proof of the equivalence of the former with the analytic K-homology theory of Kasparov. As an intermediate step towards proving our main result we use Thom's transversality theorem to describe Gysin maps compatibly with Jakob's definition of K-homology.
\end{abstract}

\maketitle

\section{Introduction}

This paper is a continuation of \cite{Part1}, which we will refer to as Part 1.
In Part 1, we studied Gysin maps (also known as wrong-way maps or transfer maps) in K-homology of Witt pseudomanifolds. To each oriented Witt pseudomanifold $M$ of dimension $n,$ we associated a class
\begin{equation*}
	\mathrm{sign}_K(M) = 2^{-\lfloor n/2 \rfloor}[D_M^{\mathrm{sign}}] \in \K^{\mathrm{an}}_n(M)[\tfrac12],
\end{equation*}
where $[D_M^{\mathrm{sign}}]$ is the class of the signature operator (of an appropriate `wedge metric'), and we showed that this assignment is an orientation of $M$ in K-homology with $2$ inverted.
(As noted in {\em loc cit}, one reason we use the signature operator of an oriented Witt space instead of the Dirac operator of a $\mathrm{Spin}^c$ stratified space is that the former is guaranteed to be Fredholm for any choice of wedge metric while the latter is not; in order for the signature operator to induce an orientation we then need to invert $2.$) Among other results, we showed that whenever $M$ and $N$ are oriented Witt spaces and 
\begin{equation*}
	f: M \lra N
\end{equation*}
is either a fiber bundle map (where we do allow for the fibers to be Witt), or a normally nonsingular inclusion then there is a Gysin map
\begin{equation*}
	f^!_{\an}: \K^{\an}_*(N)[\tfrac12] \lra \K^{\an}_*(M)[\tfrac12]
\end{equation*}
which maps the K-orientation of $N$ to that of $M$, viz.
\begin{equation*}
f^!_{\an}  \mathrm{sign}_K(N)=\mathrm{sign}_K(M).
\end{equation*}
We also showed that $\mathrm{sign}_K(M)$ is compatible with the Sullivan-Siegel orientation of $M,$ $\Delta(M) \in \KO^{\topo}_n(M)[\tfrac12].$ Specifically we showed that if $c:\KO \to \K$ denotes complexification and $\Psi^2: \KO [\tfrac12] \to \KO [\tfrac12]$ the stable second Adams operation, then under the natural isomorphism $\Phi$ between $\K^\topo_n (M)$ and $\K^\an_n (M)$ the class
\begin{equation*}
	c(\Psi^2)^{-1} \Delta (M) \in \K^\topo_n (M)[\smlhf]
\end{equation*}
corresponds to the class $\mathrm{sign}_K(M).$
 
In the present Part 2 we show that the analytic Gysin maps we defined in Part 1 are compatible with topologically defined Gysin maps. Namely we show that the map
\begin{equation*}
	\lambda = \Phi \circ c\circ (\Psi^2)^{-1}: \KO_* ^{{\rm top}}(M)[\tfrac12] \lra \K^\an_*(M)[\tfrac12]
\end{equation*}
intertwines the topological and analytic Gysin maps, 
\begin{equation*}
	\lambda\circ f^!_\topo = f^!_\an \circ \lambda
\end{equation*}
if $f$ is a normally non-singular inclusion or a fiber bundle projection. 
In contrast to Part 1 we allow $f$ to be a normally nonsingular inclusion 
of arbitrary finite CW-complexes (as opposed to Witt pseudomanifolds), 
but when $f$ is a fibration we require that the associated structure group is 
compact Lie to define the topological Gysin map. 
Via the Mostow-Palais equivariant embedding theorem, this enables
sufficient control over fiberwise embeddings into Euclidean spaces over
the base. More generally, it is often possible to work with
block bundles and  replace fiberwise with blockwise embeddings.
In that setting, a topologically Gysin map
had been defined by the second author in \cite{Ban:BTLOCSS, Ban:TSSKSS}.
That map agrees with the topological Gysin maps considered here, since
locally trivial fiber bundles are in particular block bundles, and the Thom
classes correspond under the map that forgets the projection and only
remembers the blocks.

An interesting part of our analysis is the use of a bordism description of K-homology, analogous to the geometric K-homology groups of Baum-Douglas \cite{BauDou:HIT}. Thus we make use of the theory developed by Jakob in \cite{Jak:BDH} where any multiplicative cohomology theory $h^*$ on finite CW-complexes is provided with a geometric description $h'_*$ of its dual homology theory $h_*.$ 
Classes in $h_*'(X,A)$ are represented by 4-tuples $(M,u(\nu_M), x, f)$ with $M$ a compact smooth manifold with (possibly empty) boundary, $u(\nu_M)$ an $h^*$-orientation of the stable normal bundle of $M,$ $x \in h^*(M),$ and $f:(M,\pa M) \lra (X,A)$ a continuous map. Two representatives represent the same class if they are related by diffeomorphism, bordism, or vector bundle modification. Jakob shows that $h_*'$ is naturally equivalent to $h_*.$ We apply Jakob's theory to the two cohomology theories $h^*(-) = \K^*(-)$ and $h^*(-)= \K^*(-)[\tfrac12]$ and denote the corresponding $h'_*(-)$ by 
\begin{equation*}
	\K^b_* ( - ) \quad\quad \text{and}\quad\quad \K^{b/2}_* ( - ),
\end{equation*}
respectively, where $b$ stands for {\em bordism}. 
For $\K^{b/2}_*,$ representing manifolds $M$ are required only to be oriented
over ordinary integral homology, since then the Sullivan orientation $\Delta (M)$
provides the required $\K^* [\smlhf]$-orientation away from $2$.
Using the signature operator, we define a homomorphism  
$\kappa: \K^{b/2}_* ( - ) \rightarrow \K^\an_*(-)[\tfrac12],$ analogous to the Baum-Higson-Schick homomorphism 
$\mu:\K^\geo ( - )\longrightarrow \K^\an_*  ( - )$ 
in the $\Spin^c$-case
\cite{baumhigsonschick}, 
and we show that $\kappa$ is an isomorphism.
The module structure of $h'_* ( - )$ over $\K^* ( - )$ and the $\K^* ( - )$-linearity of $\kappa$ 
play a crucial role in our proof. We use the bordism-like description of $\K_*^{\topo}[\tfrac12]$
as a tool in order to interpolate between the topological and the analytical description of K-homology.
Notice that our treatment of the isomorphism $\kappa$ constitutes an alternative to the results of Baum-Higson-Schick and has, in our opinion, an independent interest.
\\

We start out in section \ref{sect.functoriality}, establishing some naturality results in analytic K-homology concerning the interplay between Gysin maps, pull-back maps, and cup products with orientations. 
Section \ref{sect-siegel-sullivan} reviews the Siegel-Sullivan orientation and its lift by the second author to spectra. In section \ref{sec:BdismDescKHom}, we study the geometric descriptions of K-homology and K-homology with $2$ inverted. Sections \ref{sec:CompHoriented} and \ref{sec:CompKoriented} show the compatibility of Gysin maps for oriented normally nonsingular inclusions in K-homology with $2$ inverted and ordinary K-homology, respectively. Finally section \ref{sec:submersive} establishes the compatibility of Gysin maps in the submersive case.

\subsection{Acknowledgements}
The authors thank the University of Heidelberg, Sapienza Universit\`a di Roma, Stanford University, and Washington University in St. Louis for hosting research visits and are happy to acknowledge interesting conversations with Daniel Grieser, Jens Kaad, Markus Land, Eric Leichtnam, Rafe Mazzeo, Richard Melrose,  Jonathan Rosenberg, J\"org Sch\"urmann, Walter van Suijlekom, Lukas Waas, and Jon Woolf.
M. B. is funded by a research grant of the
 Deutsche Forschungsgemeinschaft (DFG, German Research Foundation)
 -- Projektnummer 495696766. This research was also partially funded by INdAM,
 Istituto Nazionale di Alta Matematica.

\subsection{Notation}
Various K-theoretic groups will play a r\^ole in the present paper and
need to be distinguished notationally.
The symbols
$\K^\topo_*$, $\K^*$, $\KO^\topo_*$ and $\KO^*$
denote topological complex K-homology,
topological complex K-theory, 
topological real K-homology and topological real K-theory.
These are represented by ring spectra $\K$ and $\KO$.
The connective versions of these spectra
will be denoted by $\ku$ and $\ko$.
The geometric complex K-homology
of Baum and Douglas \cite{BauDou:HIT} will be written as
$\K^\geo_*$. The bordism-like description of 
the homology theory associated to
$\K^* ( - )[\tfrac12]$, due to Jakob,
will be denoted by $\K^{b/2}_* ( - )$ and similarly, the bordism-type description of 
the homology theory associated to
$\K^* ( - )$ will be denoted by $\K^{b}_* ( - )$; we remind the reader that $b$ 
stands for {\em bordism}.
The Kasparov K-groups of a pair $(A, B)$ of $C^*$-algebras
are denoted by $\KK^* (A, B)$.
Analytic complex K-homology groups of a (locally compact Hausdorff) 
topological space $X$ are given by the Kasparov groups
$\K^\an_* (X) = \KK^* (C_0 (X), \cplx)$.


\tableofcontents


\section{Functoriality results}\label{sect.functoriality}
A complex vector bundle $E$ over $X$ determines an element
$[E] \in \K^0 (X)$ and also an element in $\KK_0 (\mathbb{C},C(X))$ given that 
$E$ is a finitely generated projective $C(X)$-module. 
We also have a bivariant element
\[ [[E]] \in \KK_0 (C(X), C(X)),  \] 
which is obtained by looking at $E$ as a $C(X)$-bimodule, see Blackadar
\cite[p. 259]{Bla:KOA}.
According to loc. cit., we can
also express $[[E]]$ as the external Kasparov product of $[E]\in
KK_0 (\mathbb{C},C(X))$ with the class $[\Delta_X]\in \KK_0 (C(X)\otimes C(X),C(X))$
defined by the diagonal $\Delta_X: X\to X\times X$. Here we are using the product
\[ KK_0 (\mathbb{C},C(X))\otimes \KK_0 (C(X)\otimes D, C(X))\to  KK_0 (D, C(X)), \]
with $D=C(X)$.
The abelian group $\KK_0 (C(X), C(X))$ is a unital ring under
Kasparov multiplication.
The map
\[ \K^0 (X) \longrightarrow \KK_0 (C(X), C(X))  \]
sending $[E]$ to $[[ E ]]$ is a homomorphism of rings, since
\[ [E] \cup [F] = [E\otimes F] = [[E]] \otimes [[F]].  \]

\noindent
Given more generally an element $x\in \K^1 (X)$ we also have $[[x]]\in \KK^1 (C(X),C(X))$
defined as $x\otimes [\Delta_X]$.

\noindent
Via the Kasparov product
\[ \KK_0 (C(X),C(X)) \otimes \KK_* (C(X),\cplx) \stackrel{\otimes}{\longrightarrow}  
   \KK_* (C(X),\cplx), \]
the group $\K^\an_* (X) = \KK_* (C(X), \cplx)$ becomes a module over
$\KK_0 (C(X),C(X))$ and thus over $\K^0 (X)$.
If $M$ is a $\Spin^c$-manifold and $E,E'$ complex vector bundles over $M$, 
then the identity
\begin{equation} \label{equ.dspincmeeprime} 
[[E']] \otimes [D^{\Spin^c}_{M,E}] = [D^{\Spin^c}_{M,E\otimes E'}]
  \in \K^\an_* (M)   
\end{equation}
holds for the class of the twisted $\Spin^c$-Dirac operator; see 
Blackadar \cite[Lemma 24.5.3]{Bla:KOA}
with $B=\mathbb{C}$.
Similarly, if $M$ is an oriented manifold, then
\begin{equation} \label{equ.dsignmeeprime} 
[[E']] \otimes [D^\sign_{M,E}] = [D^\sign_{M,E\otimes E'}]
  \in \K^\an_* (M);  
\end{equation}
see also Rosenberg \cite[p. 8]{rosenbergncgexampandapps}.

\begin{remark}
More generally, suppose $M$ is a smooth manifold, $E_0,$ $E_1,$ and $F$ are vector bundles over $M$ and $\nabla^F$ is a connection on $F.$
For every differential operator $D \in \mathrm{Diff}^1(M;E_0, E_1)$ there is a unique differential operator
\begin{equation*}
	D_F \in \mathrm{Diff}^1(M;E_0\otimes F, E_1\otimes F)
\end{equation*}
with the property that $D_F(s\otimes \tau)(x) = (Ds \otimes \tau)(x)$ whenever $s$ is a section of $E_0$ and $\tau$ is a section of $F$ satisfying $\nabla^F\tau=0.$
If $D$ is elliptic then both $D$ and $D_F$ define a class in $\mathrm{KK}( C_0(M), \bb C)$ and these are related by
\begin{equation*}
	[D_F] = [[F]] \otimes [D].
\end{equation*}
(The construction of $D_F$ requires a connection on $F$ but the resulting $\mathrm{KK}$-class is independent of the choice.)
\end{remark}

\begin{prop} \label{prop.kaspmultcapnaturality}
Let $f: M \to N$ be a continuous map between compact smooth manifolds.
Let $E$ be a complex vector bundle over $N$. Then the naturality relation
\[ f_* ([[f^* E]] \otimes \alpha) = [[E]] \otimes f_* \alpha 
    \in \K^\an_* (N) \]
holds for any $\alpha \in \K^\an_* (M)$.
\end{prop}
\begin{proof}
We shall use the correspondence picture of $\KK$-theory
as in Connes-Skandalis \cite[Section III]{ConSka:TLPF}, \cite{ConSka:LITF}
Recall that if $X,Y$ are locally compact spaces with $Y$ a smooth manifold,
then elements of $\KK (X,Y)$ are represented by correspondences
\[ \xymatrix@C=15pt@R=15pt{
& (Z,E_Z) \ar[dl]_{f_X} \ar[dr]^{g_Y} & \\
X & & Y,
} \]
where $Z$ is a smooth manifold, $E_Z$ is a complex vector bundle over $Z$,
$f_X$ is continuous and proper, and $g_Y$ is continuous and $\K$-oriented. More
precisely, a correspondence as above
defines naturally the element $f_* ([[E]]\otimes_Z (g_Y)!)$, where we recall that  $(g_Y)!\in \KK(Z,Y)$
and moreover every element in $\KK (X,Y)$ arises in this way.
The Kasparov multiplication of two correspondences
\[ \xymatrix@C=15pt@R=15pt{
& (Z_1,E_1) \ar[dl]_{f_X} \ar[dr]^{g_M} & \\
X & & M
} 
\text{ and }
\xymatrix@C=15pt@R=15pt{
& (Z_2,E_2) \ar[dl]_{f_M} \ar[dr]^{g_Y} & \\
M & & Y
} \]
is given by the diagram
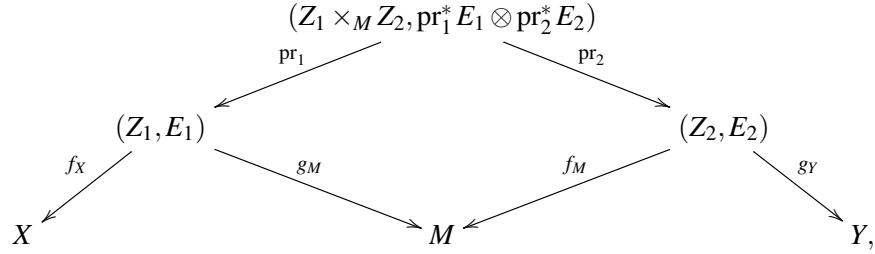
\[ \xymatrix{
& & (Z_1 \times_M Z_2, \pro^*_1 E_1 \otimes \pro^*_2 E_2) 
  \ar[ld]_{\pro_1} \ar[rd]^{\pro_2} & & \\
& (Z_1,E_1) \ar[ld]_{f_X} \ar[rd]^{g_M} & & (Z_2,E_2) \ar[ld]_{f_M} \ar[rd]^{g_Y} & \\
X & & M & & Y,
} \]
where one assumes without loss of generality that
$g_M$ and $f_M$ are transverse to each other.

The element $[E] \in \K^0 (N) = \KK (\cplx, C(N))$ is represented
by the correspondence
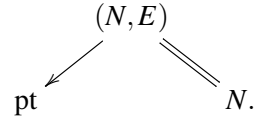
\[ \xymatrix@C=15pt@R=15pt{
& (N,E) \ar[dl] \ar@{=}[dr] & \\
\pt & & N.
} \]
Here we are using that $N$ is a smooth manifold.
Note that the identity map is canonically $\K$-oriented.
We claim that the bivariant element $[[E]] \in \KK (N,N)$ is represented
by the correspondence
\begin{equation} \label{equ.brbrecorr}
\xymatrix@C=15pt@R=15pt{
& (N,E) \ar@{=}[dl] \ar@{=}[dr] & \\
N & & N.
} \end{equation}
This can for example be formally deduced from the identity
\[ [[E]] = [E] \otimes [\Delta_N]  \] 
(Blackadar \cite[p. 259]{Bla:KOA}), where
$[\Delta_N] \in \KK (C(N)\otimes C(N), C(N))$ is the class
defined by the diagonal $\Delta_N: N \to N\times N$, and
$\otimes$ is the external Kasparov product
\[ \otimes: \KK_0 (\cplx,C(N))\otimes \KK_0 (C(N)\otimes D, C(N))\to  KK_0 (D, C(N)) \]
with $D=C(N)$.
Indeed, in general the external Kasparov product of two correspondences
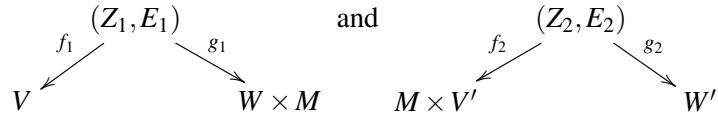
\[ \xymatrix@C=15pt@R=15pt{
& (Z_1,E_1) \ar[dl]_{f_1} \ar[dr]^{g_1} & \\
V & & W \times M
} 
\text{ and }
\xymatrix@C=15pt@R=15pt{
& (Z_2,E_2) \ar[dl]_{f_2} \ar[dr]^{g_2} & \\
M \times V' & & W'
} \]
is given by the diagram
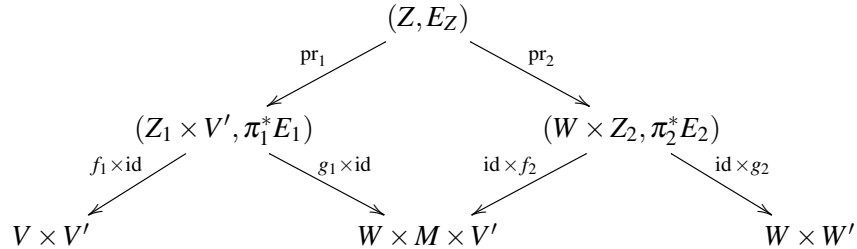
\[ \xymatrix@C=10pt{
& & (Z, E_Z) 
  \ar[ld]_{\pro_1} \ar[rd]^{\pro_2} & & \\
& (Z_1 \times V', \pi^*_1 E_1) \ar[ld]_{f_1 \times \id} \ar[rd]^{g_1 \times \id} 
  & & (W \times Z_2, \pi^*_2 E_2) \ar[ld]_{\id \times f_2} \ar[rd]^{\id \times g_2} & \\
V\times V' & & W\times M \times V' & & W\times W'
} \]
in the transverse situation, with
$Z := (Z_1 \times V') \times_{W\times M\times V'} (W \times Z_2),$
$E_Z := \pro^*_1 \pi^*_1 E_1 \otimes \pro^*_2 \pi^*_2 E_2$,
$\pi_1: Z_1 \times V' \to Z_1,$ $\pi_2: W\times Z_2 \to Z_2$.
Applying this with
$V := \pt$,
$V' := N$,
$W := \pt$,
$W' := N$,
$M := N$,
$Z_1 := N$,
$E_1 := E$,
$Z_2 := N$ and
$E_2 := 1$,
to the correspondences
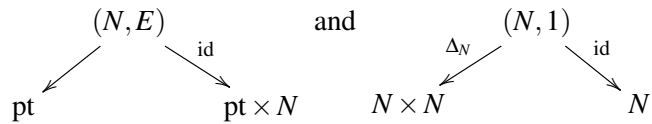
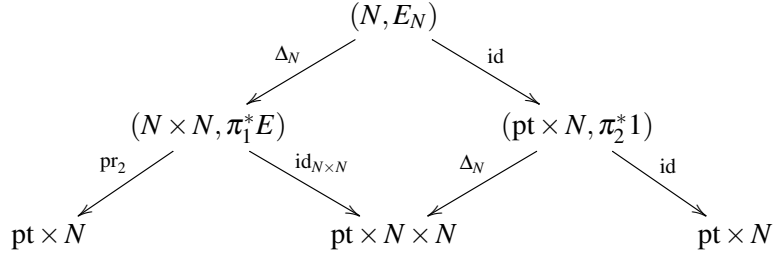
\[ \xymatrix@C=15pt@R=15pt{
& (N,E) \ar[dl] \ar[dr]^{\id} & \\
\pt & & \pt \times N
} 
\text{ and }
\xymatrix@C=15pt@R=15pt{
& (N,1) \ar[dl]_{\Delta_N} \ar[dr]^{\id} & \\
N \times N & & N
} \]
yields
\[ \xymatrix@C=10pt{
& & (N, E_N) 
  \ar[ld]_{\Delta_N} \ar[rd]^{\id} & & \\
& (N \times N, \pi^*_1 E) \ar[ld]_{\pro_2} \ar[rd]^{\id_{N\times N}} 
  & & (\pt \times N, \pi^*_2 1) \ar[ld]_{\Delta_N} \ar[rd]^{\id} & \\
\pt \times N & & \pt \times N \times N & & \pt \times N
} \]
with
\[ E_N = \Delta_N^* \pi^*_1 E \otimes \id^*_2 \pi^*_2 1 = E \]
($\pi_1: N \times N \to N,$ $\pi_2: \pt \times N \to N$).
Since $\pro_2 \circ \Delta_N = \id_N$, this is nothing but
(\ref{equ.brbrecorr}), as claimed.
The given element $\alpha \in \KK (M,\pt)$ is represented by a
correspondence
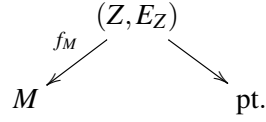
\[ \xymatrix@C=15pt@R=15pt{
& (Z,E_Z) \ar[dl]_{f_M} \ar[dr] & \\
M & & \pt.
} \]
The map $f:M\to N$ induces a map $f^\sharp: C(N) \to C(M)$, which
determines an element 
$[f^\sharp] \in \KK (C(N),C(M))$.
This element is represented by the correspondence
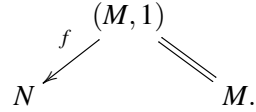
\[ \xymatrix@C=15pt@R=15pt{
& (M,1) \ar[dl]_{f} \ar@{=}[dr] & \\
N & & M.
} \]
(Note that $f$ is indeed proper.)
By definition,
\[ f_* = [f^\sharp] \otimes -: \KK (M,\pt) \longrightarrow \KK (N,\pt).  \]
In terms of correspondences, this implies that 
$f_* (\alpha) \in \KK (N,\pt)$ is represented by
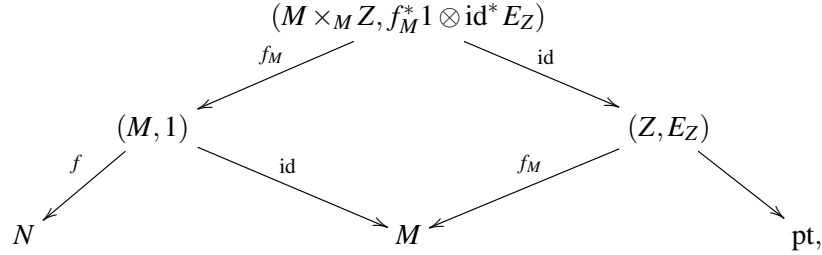
\[ \xymatrix{
& & (M \times_M Z, f^*_M 1 \otimes \id^* E_Z) 
  \ar[ld]_{f_M} \ar[rd]^{\id} & & \\
& (M,1) \ar[ld]_{f} \ar[rd]^{\id} & & (Z,E_Z) \ar[ld]_{f_M} \ar[rd] & \\
N & & M & & \pt,
} \]
that is, by
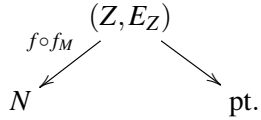
\[ \xymatrix@C=15pt@R=15pt{
& (Z,E_Z) \ar[dl]_{f\circ f_M} \ar[dr] & \\
N & & \pt.
} \]
Consequently, $[[E]] \otimes f_* (\alpha)$ is represented by
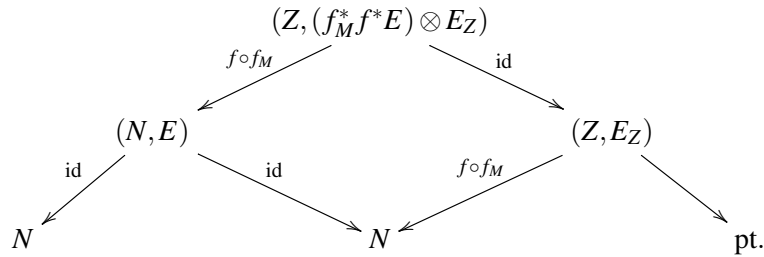
\begin{equation} \label{equ.brbreotimesfstalpha}
\xymatrix{
& & (Z, (f^*_M f^* E) \otimes E_Z) 
  \ar[ld]_{f\circ f_M} \ar[rd]^{\id} & & \\
& (N,E) \ar[ld]_{\id} \ar[rd]^{\id} & & (Z,E_Z) \ar[ld]_{f\circ f_M} \ar[rd] & \\
N & & N & & \pt.
} \end{equation}

On the other hand,
$[f^* E] \in \KK (\cplx, C(M))$ is represented by
\[ \xymatrix@C=15pt@R=15pt{
& (M,f^* E) \ar[dl] \ar@{=}[dr] & \\
\pt & & M
} \]
and thus $[[f^* E]] \in \KK (M,M)$ by
\[ \xymatrix@C=15pt@R=15pt{
& (M,f^* E) \ar@{=}[dl] \ar@{=}[dr] & \\
M & & M.
} \]
Therefore, $[[f^* E]] \otimes \alpha$ is given by
\[ \xymatrix{
& & (Z, f^*_M (f^* E) \otimes E_Z) 
  \ar[ld]_{f_M} \ar[rd]^{\id} & & \\
& (M,f^* E) \ar[ld]_{\id} \ar[rd]^{\id} & & (Z,E_Z) \ar[ld]_{f_M} \ar[rd] & \\
M & & M & & \pt.
} \]
Now, as discussed above, $f_* ([[f^* E]] \otimes \alpha)$ is obtained from 
this correspondence by composing the left hand arrow with $f$.
This gives precisely the correspondence (\ref{equ.brbreotimesfstalpha}), which represents
$[[E]] \otimes f_* (\alpha)$. 
\end{proof}

The above result will be used later, when we shall prove the
 $\K^0 ( - )$-linearity of certain maps in K-homology.
 
 \subsection{Functoriality properties of analytic Gysin maps for normally non-singular  inclusions}
Let $h$ be a ring spectrum defining a multiplicative cohomology theory $h^*$
and a homology theory $h_*$. Then homological Gysin maps $j^!: h_* (X)\to h_{*-r} (Y)$
associated to a map $j: Y \to X$
generally interact with the cap product (at least up to sign) by
\[ j^! (\xi \cap a) = (j^* \xi) \cap j^! (a),~
  \xi \in h^* (X),~ a \in h_* (X)   \]
see e.g. Boardman \cite[Chapter V, p. 35, 6.2]{boardman}. An analogous relation is
therefore expected to hold in analytic theory as well. We formulate this in the
next lemma.
\begin{lemma} \label{lem.gysinandcap}
Let $M$ be a K-oriented  smooth manifold. Let $E$ be a complex vector bundle over $M$.
Assume that $j: N\subset M$ is a K-oriented inclusion. Then \[ 
 j^!_{\an,\Spin^c} ([[E]] \otimes [D^{\Spin^c}_M]) = [[j^* E]] \otimes j^!_{\an,\Spin^c} [D^{\Spin^c}_M]\quad\text{in}
 \quad \K_*^\an (N)\,.
 \]
\end{lemma}
\begin{proof}
We are interested in the equality
 $j^!_{\an,\Spin^c} ([[E]] \otimes [D^{\Spin^c}_M]) = [[j^* E]] \otimes j^!_{\an,\Spin^c} [D^{\Spin^c}_M]$.
 
 \medskip
 \noindent
 We use $ j^!_\an $ instead of  $j^!_{\an,\Spin^c}$ and $D_M$ instead of $D^{\Spin^c}_M$.
 
  \medskip
 \noindent
Since $j^* E= E|_N$ and $j^!_\an [D_M] =[D_N]$ this means that we want to prove
 the following equality:
\[  j^!_\an ([[E]] \otimes [D_M]) =[D_{N,E|_N}]  \]
that is, 
\[  j^!_\an ([D_{M,E}]) =[D_{N,E|_N}]  \,.\]
Let $\mathcal{U}\hookrightarrow M$ be a tubular neighbourhood of $N$ and let $\varphi: \mathcal{U}\to V$ be a diffeomorphism with an orientable real vector bundle $V\xrightarrow{\pi} N$, the normal bundle of $N$ in $M$. Then, by definition 
\begin{equation}\label{useful}j^!_\an ( - ):= \Sigma(\pi)^{-1} \otimes \varphi ! \otimes i !\otimes -
\end{equation}
where 
\begin{itemize}
\item 
$i!\in \KK^*(C_0 ( U),C (M))$ is the bivariant class that defines $(i^\sharp)^* : \KK^*(C(M),\mathbb{C}) \to \KK^* (C_0 (U),\mathbb{C})$ with $i^\sharp: C_0 (U) \to C(M)$ being equal to extension by zero;
\item $\varphi !\in \KK^*(C_0 ( V),C_0  (U))$ is the bivariant element inducing the isomorphism\\
 $(\varphi^*)^*\equiv \varphi_*: \KK^* (C_0 (\mathcal{U} ), \cplx) \to \KK^* (C_0 (V), \cplx)$
 \item $\Sigma (\pi)^{-1}\in \KK^c (C (N), C_0(V))$ is the inverse of $\Sigma (\pi) \in \KK^c (C_0 (V), C(N))$ with $c$ the codimension of $N$ in $M$,
 that is the rank of $V$. In the $\Spin^c$ case considered here 
 we do not need to invert 2; recall that in this case $\Sigma (\pi)$ is induced
 by a family of $\Spin^c$-Dirac operators. 
 \end{itemize}
 Now, we know that
 $$ i ! \otimes [D_{M,E}] = (i^\sharp)^*  [D_{M,E}] =  [D_{\mathcal{U},E|_{\mathcal{U}}}]$$
Moreover, jazzing up our proof of the 
 diffeomorphism invariance of the signature class, we can check that 
 $$  \varphi ! \otimes [D_{\mathcal{U},E|_{\mathcal{U}}}] \equiv \varphi_* [D_{\mathcal{U},E|_{\mathcal{U}}}]=
   [D_{V,E_V}] $$
 with $E_V$ the bundle induced on $V$ by $\varphi$ and $E|_{\mathcal{U}}$. Notice that by homotopy invariance
 $[D_{V,E_V}] $ is equal to $[D_{V,\pi^* (E|_N)}]$ since the restriction to the zero-section of
 $E_V$ is $E|_N$. Thus, we are reduced to prove that
 $$  \Sigma(\pi)^{-1} \otimes ([D_{V,\pi^* (E|_N)}])= [D_{N,E|_N}]$$
 or, equivalently, that 
 $$[D_{V,\pi^* (E|_N)}]=\Sigma(\pi) \otimes ( [D_{N,E|_N}])$$
In the present $\Spin^c$ case these 3 elements are given by correspondences, as in Connes-Skandalis.
In particular, $\Sigma (\pi)$ is given by the correspondence
\[ \xymatrix@C=15pt@R=15pt{
& (V,{\bf 1}) \ar[dl]_{{\rm Id}} \ar[dr]^{\pi} & \\
V & & N,
} \]
where ${\bf 1}$ denotes the trivial bundle; on the other hand $ [D_{N,E|_N}]$
is given by
\[ \xymatrix@C=15pt@R=15pt{
& (N,E|_N) \ar[dl]_{{\rm Id}} \ar[dr]^{p} & \\
N & & {\rm point},
} \]
and it easy to see that the composition of these two correspondences, in the above order, is the correspondence
that defines $[D_{V,\pi^* (E|_N)}]$.
\end{proof}

\noindent
We shall  also and mainly need the following version of the above Lemma, where we only assume orientability.
Let us recall the notation, introduced in Part 1,
\begin{equation}\label{notation-sign}
\sign_K (M):= 2^{-\lfloor n/2 \rfloor} [D^{{\rm sign}}_{M}]\in \K_n^\an (M)[\smlhf],
\end{equation}
where $n=\dim M$.
If  $M$ and $N$ are  oriented smooth manifolds and  $j: N\subset M$ is a smooth oriented immersion,
then Hilsum has proved in \cite{Hil:FKBPVL} that
\begin{equation} \label{equ.hilsumgysinrestrofsignk}
j^!_\an \sign_K (M)= \sign_K (N)
\end{equation}
with $j^!_\an$ given by left Kasparov multiplication by the element $\Sigma(j)\in \KK^\ell (C(M), C(N))[\smlhf]$
given by the right hand side of \eqref{useful}, with $\ell$ equal to the codimension of the immersion.
(Needless to say, in Part 1 we have extended this result to normally non-singular inclusions of Witt spaces.)

\begin{lemma} \label{lem.gysinandcap-bis}
Let $M$ and $N$ be  oriented smooth manifolds and let $j: N\subset M$ be an oriented smooth
immersion. Let $E$ be a complex vector bundle over $M$. Then
\[  j^!_\an ([[E]] \otimes \sign_K (M)) = [[j^* E]] \otimes j^!_\an \sign_K (M) \quad\text{in}\quad \K_*^\an (N)[\smlhf]\,.\]
\end{lemma}

\begin{proof}
Correspondences 
\[ \xymatrix@C=15pt@R=15pt{
& (Z,E) \ar[dl]_{f_X} \ar[dr]^{g_Y} & \\
X & & Y,
} \]
where we more generally consider $g_Y$ oriented have been considered by Hilsum in \cite[Section 4.4]{Hil:FKBPVL}.
For these more general correspondences we need to work with $\KK^* (X,Y)[\smlhf]$; if we work 
with this group and the signature operator then the same results of Connes-Skandalis hold, thanks to the
functoriality results established by Hilsum. Thus the proof of the Lemma \ref{lem.gysinandcap} can be run in this more general
setting giving the equality $j^!_\an ([[E]] \otimes \sign_K (M)) = [[j^* E]] \otimes j^!_\an \sign_K (M)]$
in  $\K_*^\an (N)[\smlhf]$.
\end{proof}

Another very useful result, inspired by a result by Boardman in the topological context, is the following one.
Let $j:M\to S$ be an {\em oriented} immersion of oriented manifolds. Let $\ell$ be the codimension of this
immersion. As already recalled, we then have an element $\Sigma(j)\in \KK^\ell (C(M), C(S))[\smlhf]$;  
{\bf right} Kasparov multiplication by this element defines a homomorphism
\begin{equation}\label{lower-shriek-oriented}
j_!: \K^* (M)[\smlhf]\rightarrow \K^{*+\ell} (S)[\smlhf]\,,\quad j_! (x):= x\otimes \Sigma (j)\,.
\end{equation}
which extends to the oriented case the well-known homomorphism in K-cohomology defined for K-oriented
immersions.

\begin{prop}\label{prop-boardman}
Let $x\in \K^0 (M)[\smlhf]$. Let $j:M\to S$ be an oriented immersion of oriented manifolds. Then
 \begin{equation} \label{boardman-analytic-pre}
 [[j_! (x)]] \otimes \sign_K (S)
  = j_* ([[x]] \otimes j^! \sign_K (S))\end{equation}
\end{prop}

\begin{proof}
Let $x\in \K^0 (M)[\smlhf]$. Recall that $[[x]]=x\otimes [\Delta_M]$ with $[\Delta_M]\in \KK_0 (C(M)\otimes C(M),C(M))$
defined by the diagonal $\Delta_M: M\to M\times M$. 
Similarly, $[[j_! (x)]]= j_! (x)\otimes  [\Delta_S]$. 
We shall
prove \eqref{boardman-analytic-pre} using correspondences defined by oriented maps, as in \cite[Section 4.4]{Hil:FKBPVL}. As the arguments are similar to the ones establishing Proposition \ref{prop.kaspmultcapnaturality} we shall not spell out all the details.

\noindent
We rewrite \eqref{boardman-analytic-pre} as
$$(j_! (x)\otimes  [\Delta_M]) \otimes \sign_K (S)=  j_* ((x\otimes [\Delta_M]) \otimes j^! \sign_K (S))\,.$$
Let $p: M\to {\rm pt}$ the map to a point. Consider $x\in \K^0 (M)$; it is given by
\[ \xymatrix@C=15pt@R=15pt{
& (M,x) \ar[dl]_{p} \ar[dr]^{{\rm id}} & \\
{\rm pt} & & M,
} \]
whereas $[[x]]$, which is  $x\otimes [\Delta_M]\in \KK^0 (C(M),C(M))$, is given by 
\[ \xymatrix@C=15pt@R=15pt{
& (M,x) \ar[dl]_{{\rm id}} \ar[dr]^{{\rm id}} & \\
M & & M.
} \]
The element $\Sigma (j)\in \KK^\ell (C(M), C(S))[\smlhf]$ is given by the correspondence
\[ \xymatrix@C=15pt@R=15pt{
& (M,1) \ar[dl]_{{\rm id}} \ar[dr]^{j} & \\
M & & S.
} \]
and thus it is not difficult to see, using the composition of correspondences, that $x\otimes  \Sigma (j)$ is given by
\[ \xymatrix@C=15pt@R=15pt{
& (M,x) \ar[dl]_{p} \ar[dr]^{j} & \\
{\rm pt} & & S
} \]
and that 
 $[[x\otimes  \Sigma (j)]]$, which is $[[j_! (x)]]$, is given by
\[ \xymatrix@C=15pt@R=15pt{
& (M,x) \ar[dl]_{j} \ar[dr]^{j} & \\
S & & S.
} \]
Since $\sign_K (S)\in \K^\an_* (S)$ is given by
\[ \xymatrix@C=15pt@R=15pt{
& (S,1) \ar[dl]_{{\rm id}} \ar[dr]^{p} & \\
S & & {\rm pt}.
} \]
we conclude, once again with (easy) arguments involving the composition of correspondences, that the left hand side of  \eqref{boardman-analytic-pre} 
is given by the correspondence
\begin{equation}\label{leftboardman}
 \xymatrix@C=15pt@R=15pt{
& (M,x) \ar[dl]_{j} \ar[dr]^{p} & \\
S & & {\rm pt}.
} 
\end{equation}
Now, similarly, $\Sigma(j)\otimes \sign_K (S)$ is given by
\[ \xymatrix@C=15pt@R=15pt{
& (M,1) \ar[dl]_{{\rm id}} \ar[dr]^{p} & \\
M & & {\rm pt}.
} \]
so that $[[x]]\otimes (\Sigma(j)\otimes \sign_K (S))$ is given by
\[ \xymatrix@C=15pt@R=15pt{
& (M,x) \ar[dl]_{{\rm id}} \ar[dr]^{p} & \\
M & & {\rm pt}.
} \]
When we apply $j_*$, which  is also given by a Kasparov product, we finally obtain that the right hand side 
of  \eqref{boardman-analytic-pre}  is given by 
\begin{equation*}
 \xymatrix@C=15pt@R=15pt{
& (M,x) \ar[dl]_{j} \ar[dr]^{p} & \\
S & & {\rm pt}
} 
\end{equation*}
and since this is the same as \eqref{leftboardman} we conclude that the proof of the Proposition is now
complete.
\end{proof}

\section{Siegel-Sullivan orientation classes.}\label{sect-siegel-sullivan}

\noindent
We recall basic material on the Siegel-Sullivan orientation
from \cite{Sul:GTP}, \cite{Sie:WSGCTKOP} and \cite{Ban:TSSKSS}.
The first of these references constructs the class for vector bundles
(in fact for disc block bundles) and for PL manifolds,
the second extends this to Witt pseudomanifolds, and the third
lifts this extension to the ring spectrum level,
which led to a number of transfer homomorphism results.
In \cite[Cor. 6.4, p. 200]{Sul:GTP}, Sullivan
associates to every oriented vector bundle $\xi$
(or even to an oriented PL block bundle) of rank $d$ over a compact polyhedron $B$
a Thom class
\[ \Delta_\SO (\xi) \in \widetilde{\KO}^d (\Th \xi)[\smlhf]. \]
This element is determined by the signature invariant
\[
\bigoplus_{j=0}^\infty \Omega^\SO_{d+4j} (D\xi, S\xi)\otimes \rat
 \longrightarrow \rat,
\]
of $\xi$ given by
\[
[f:(M^{d+4j},\partial M)\to (D\xi, S\xi)] \mapsto \sigma (f^{-1} (B)),
\]
together with analogously defined $\rat/\intg$-invariant, see farther below.
Here, $f$ has been made transverse to $B$ in the disc bundle $D\xi$.
(Although $B$ is not assumed to be a manifold, transversality
still works, see e.g. the remarks of Browder \cite[II.2, p. 33f]{browder}
based on Thom's transversality theorem.) 
Then the preimage $f^{-1} (B)$ is a compact manifold of
dimension 
\[ \dim M + \dim B - \dim (D\xi) = 
   (d+4j) + \dim B - (d+\dim B) = 4j.
\]
It is oriented, since $M$ is oriented, and the normal bundle
of $f^{-1} (B)$ is oriented, being the pullback of the oriented
bundle $\xi$.
Then $\sigma (f^{-1} (B))\in \intg$ denotes the signature of this
$4j$-dimensional oriented compact manifold.
For the $\rat/\intg$-part, one uses analogously the $\intg/_k$-signature
of $\intg/_k$-manifolds that represent bordism with 
$\intg/_k$-coefficients.
If $\xi = \gamma$ is the canonical bundle over the Grassmannian
$B = \BSO_{4n},$ then $\Delta_\SO (\gamma)$ has a description
as
\[ \frac{\Lambda^+ - \Lambda^-}{{\Lambda^+ + \Lambda^-}}, \]
where $\Lambda$ denotes the exterior algebra
of $\gamma$ regarded as a representation of $\SO_{4n}$
and $\Lambda_+, \Lambda_-$ are the $\pm 1$-eigenspaces
of the involution given by Clifford multiplication by the volume element.
If $M$ is a closed oriented $n$-dimensional manifold with stable normal bundle
$\nu_M,$ then $\Delta_\SO (\nu_M)$ corresponds under
Alexander duality to an element $\Delta_\SO (M) \in \KO_n (M)[\smlhf]$,
which is an orientation.
Let $X$ be a closed Witt space of dimension $n$.
Drawing on Sullivan's methods,
Siegel constructs in \cite{Sie:WSGCTKOP} a canonical orientation class
\[ \mu_X \in \KO_n (X) \otimes \ism.  \]
(In fact, the class lives in connective $\KO$-homology.)
We shall refer to $\mu_X$ as the 
\emph{Siegel-Sullivan orientation class} of $X$.
Let us briefly outline Siegel's construction, which rests on two fundamental
facts due to Sullivan:
First, there is an exact sequence
\[ 
0 \to \KO^i (Y,B)\otimes \ism \longrightarrow
  \KO^i (Y,B)^\wedge \oplus \KO^i (Y,B)\otimes \rat
  \longrightarrow \KO^i (Y,B)^\wedge \otimes \rat \to 0, 
\]  
where $\KO^i (Y,B)^\wedge$ denotes the profinite completion of $\KO^i (Y,B)$
with respect to groups of odd order.
Second, the natural transformation 
$\Delta_{\SO *}: \Omega^\SO_i (Y,B) \to (\KO [\smlhf])_i (Y,B)$
induces a Conner-Floyd type isomorphism
\[ \Omega^\SO_{i+4*} (Y,B) \otimes_{\Omega^\SO_* (\pt)} \ism
   \cong \KO_i (Y,B) \otimes \ism \]
of $\intg/4\intg$-periodic theories for compact PL pairs $(Y,B)$,
\cite{Sul:GTP}, \cite[p. 85]{madsenmilgram}.
Together with universal coefficient considerations, these two facts 
imply that elements of 
$\KO^i (Y,B)\otimes \ism$ are pairs $(\sigma_0, \tau_0)$
of homomorphisms
$\sigma_0: \Omega^\SO_{i+4*} (Y,B) \otimes \rat \to \rat$
and 
$\tau_0: \Omega^\SO_{i+4*} (Y,B;\rat/\ism) \to \rat/\ism$  
such that the \emph{periodicity relations}
\begin{equation} \label{equ.sigmatauperiodicityrelations}
\sigma_0 ([f][M\to \pt]) = \sigma (M)\cdot \sigma_0 [f],~
\tau_0 ([f][M\to \pt]) = \sigma (M)\cdot \tau_0 [f]
\end{equation}
with respect to multiplication by a closed manifold $M$ hold and the diagram
\[ \xymatrix{
\Omega^\SO_{i+4*} (Y,B) \otimes \rat \ar[d] \ar[r]^>>>>>>>>{\sigma_0} 
   & \rat \ar[d] \\
\Omega^\SO_{i+4*} (Y,B;\rat/\ism) \ar[r]^>>>>>{\tau_0} & \rat/\ism
} \]
commutes.
To define $\mu_X$ for a closed Witt space $X^n$,
choose a PL embedding $X\subset \real^m$, $m$ large, of codimension $4k$.
Let $(N,\partial N)$ be a regular neighborhood of $X$.
Since $\real^m$ is a manifold, $(N,\partial N)$ is a (compact) PL manifold
with boundary (\cite[p. 34, Prop. 3.10]{rourkesanderson}).
We will describe an element in
$\KO^{4k} (N,\partial N)\otimes \ism,$ which corresponds under 
Alexander-Spanier-Whitehead
duality to $\mu_X \in \KO_n (X)\otimes \ism$.
Therefore, we need to specify homomorphisms
$(\sigma_X, \tau_X)$ satisfying the above periodicity relations and
the integrality condition, i.e. commutativity of
\[ \xymatrix{
\Omega^\SO_{4(k+*)} (N,\partial N) \otimes \rat 
          \ar[d] \ar[r]^>>>>>>>>{\sigma_X} & \rat \ar[d] \\
\Omega^\SO_{4(k+*)} (N,\partial N;\rat/\ism) 
   \ar[r]^>>>>>{\tau_X} & \rat/\ism.
} \]
The homomorphism $\sigma_X$ is 
\[ \sigma_X ([(M,\partial M) \stackrel{f}{\longrightarrow} (N,\partial N)]
  \otimes r)
    := \sigma (\widetilde{f}^{-1} (X)) \otimes r,~ r\in \rat, \]
where one uses the block-transversality results of
\cite{buonrs}, \cite{mccrory} to make $f$ transverse to $X$
in the PL manifold $N$. The preimage $\widetilde{f}^{-1} (X) \subset M$
under the transverse map $\widetilde{f}$ has the same local structure
as $X$ and thus is again a Witt space with a well-defined 
signature $\sigma (\widetilde{f}^{-1} (X)) \in \intg$.
The homomorphism $\tau_X$ is obtained by specifying a sequence 
of homomorphisms
\[ \tau_{X,k}: \Omega^\SO_* (N,\partial N; \intg/_k) 
   \longrightarrow \intg/_k,~ k \text{ odd}, \]
compatible with respect to divisibility,
which are defined in much the same way as $\sigma_X$, but
using oriented $\intg/_k$-manifolds to represent elements of
$\Omega^\SO_* (N,\partial N; \intg/_k)$. By Novikov additivity
for the signature of compact Witt spaces with boundary, the
transverse inverse image of $X$ in the
$\intg/_k$-manifold has a well-defined (and bordism invariant)
signature in $\intg/_k$, which defines $\tau_{X,k}$.
The periodicity and integrality conditions are satisfied
and thus an element $\mu_X$ is obtained.
The homomorphism
$c_*: \KO_n (X)\otimes \ism \to \KO_n (\pt) \otimes \ism$
induced by the constant map $c: X\to \pt$ sends 
$\mu_X$ to the signature of $X$.
Using the orientation class $\mu_X$, Siegel obtains a natural
transformation 
\[ \mu^\Witt: \Omega^\Witt_* (-) \longrightarrow \KO [\smlhf]_* (-) \]
of homology theories by setting
\[ \mu^\Witt ([X \stackrel{f}{\longrightarrow} Y]) 
   = f_* (\mu_X). \]
This transformation then reduces to the signature homomorphism
on coefficient groups. In terms of the transformation, the orientation
class can of course be recovered as
\[ \mu_X = \mu^\Witt ([\id_X]). \]
Siegel's transformation factors through the homomorphism induced
by the connective cover $\ko [\smlhf] \to \KO [\smlhf],$ since
$\Omega^\Witt_* (-)$ is connective.
He shows that the natural transformation
\[ \mu^\Witt [\smlhf]: \Omega^\Witt_* (-) \otimes \ism
   \longrightarrow \ko [\smlhf]_* (-) \otimes \ism \]
is an equivalence of homology theories.
Furthermore, if $X=M$ is a
smooth compact manifold, then $\mu_M$ agrees with the Sullivan
orientation $\Delta_\SO (M)$,
\[ \mu_M = \Delta_\SO (M). \]
Let $\MWITT$ denote the (Quinn-type) ring spectrum associated to the
multiplicative ad-theory of Witt spaces, representing Witt bordism
(see Banagl-Laures-McClure \cite{BanLauMcC:LFCISNC}),
and let
\[ \Delta: \MWITT \longrightarrow \KO [\smlhf] \]
denote the second author's ring-spectrum level Siegel-Sullivan orientation,
\cite{Ban:TSSKSS}. It restricts to the Sullivan
orientation $\Delta: \MSPL \to \KO [\smlhf]$ under the canonical
map $\MSPL \to \MWITT$.
By \cite[Prop. 5.7]{Ban:TSSKSS}, the induced natural transformation
$\Delta_*: \Omega^\Witt_* (-) \to \KO_* (-)[\smlhf]$
of homology theories agrees with Siegel's transformation $\mu^\Witt$
as described above.
The homotopy ring of the complex spectrum $\K$ is $\pi_* (\K)=\intg [\beta^{\pm 1}],$
where $\beta$ is the complex Bott element in degree $2$, i.e.
$\beta$ is represented by the reduced canonical complex line bundle
$H-1 \in \widetilde{\K}^0 (S^2)$.
On $\pi_4$, the complexification $c:\KO \to \K$ induces multiplication by $2$,
$c_* =2: \pi_4 (\KO) = \intg \to \intg = \pi_4 (\K)$.
Thus there does not exist an element in $\pi_4 (\KO)$ that maps to
$\beta^2$. But after inverting $2$, such an element exists.
Let $a\in \pi_4 (\KO)[\smlhf]$ be the unique element with
$c_* (a) = \beta^2.$
Then on homotopy groups, $\Delta$ induces the homomorphism
\[ \Delta_*: \Omega^\Witt_{4k} = \MWITT_{4k} \longrightarrow
   \KO [\smlhf]_{4k} = Z \langle a^k \rangle \]
given by
\begin{equation} \label{equ.siegelsullorientoncoeffs} 
\Delta_* [V^{4k} \to \pt] = \sigma (V)\cdot a^k,
\end{equation}
where $\sigma (V)\in \intg$ is the signature of the intersection form
on the intersection homology groups $IH_{2k} (V;\rat)$ of $V$.
The Siegel-Sullivan orientation class of a compact $n$-dimensional Witt space
$(V,\partial V)$ is given by the image
\[ \Delta (V) := \Delta_* [\id_V] \in \KO_n (V,\partial V)[\smlhf]  \]
of the Witt bordism class of the identity on $V$.

We showed in part 1 (\cite[Thm. 7.6]{Part1}; the special case
of a manifold is Thm. 7.2 there)
that if $c:\KO \to \K$ denotes complexification and 
$\Psi^2: \KO [\tfrac12] \to \KO [\tfrac12]$ the stable second Adams operation, 
then the class
$c(\Psi^2)^{-1} \Delta (M) \in \K^\topo_n (M)[\smlhf]$
corresponds to $\mathrm{sign}_K(M).$
In light of this, we adopt the following definitions.
For an oriented vector bundle $\xi$ of rank $d$ over a space $X$, we set
\[ \Delta_\cplx (\xi) := c (\Psi^2)^{-1} \Delta_\SO (\xi) 
   \in \K^d (\Th (\xi),\infty)
   \otimes \intg [\smlhf], \]
and for an oriented compact manifold 
$(M,\partial M)$, we set
\[ \Delta_\cplx (M,\partial M) := c (\Psi^2)^{-1} \Delta_\SO (M,\partial M) 
   \in \K^\topo_n (M,\partial M)
   \otimes \intg [\smlhf]. \]
If $M$ has empty boundary, we will briefly write $\Delta_\cplx (M)$ for
$\Delta_\cplx (M,\varnothing)$.

\begin{lemma}\label{lemma:delta-complex-orientation}
The class $\Delta_\cplx (\xi)$ is a $\K\smlhf$-orientation for the
vector bundle $\xi$.
\end{lemma}
\begin{proof}
Sullivan's class $\Delta_\SO (\xi)$ is a $\KO \smlhf$-orientation for $\xi$.
The morphisms $\Psi^2: \KO\smlhf \to \KO\smlhf$ 
(stable Adams operation)
and $c:\KO \to \K$ (complexification)
are both ring morphisms of ring spectra.
The statement then follows from standard results, e.g.
\cite[p. 305, Prop. V.1.6]{Rud:TSOC}.
\end{proof}

\section{Bordism-like description of K-homology and the isomorphism
with $\K^\an_* ( - )$} \label{sec:BdismDescKHom}

A significant r\^ole in our arguments will be played by a bordism-like description of K-homology due to Jakob \cite{Jak:BDH} which is closely related to the geometric K-homology theory of Baum and Douglas. In this section we will recall Jakob's definition, apply it to both ordinary K-theory and K-theory with $2$ inverted, and establish isomorphisms with both analytic K-homology and geometric K-homology.

\subsection{Bordism-like description of $\K_* ( - )$ and  $\K_* ( - )[\tfrac12]$ }

\noindent
Recall that given a multiplicative cohomology theory $h^*$, represented by a 
ring spectrum $h$, Jakob 
provides a {\em bordism-type  description} $h'_*$ 
of the associated homology theory $h_*$ where a {\em description} means more precisely that there exists an explicit  isomorphism $\varphi : h'_* (X)\to h_* (X)$ 
with interesting functoriality properties.
We shall also refer to 
$h'_*$ as the {\it geometric description} of $h_*$.

Let $\Th (\xi)$ denote the Thom spectrum of a stable vector bundle $\xi$.
Recall that an \emph{$h^*$-Thom class} (or $h$-\emph{orientation}) of $\xi$
is a class $u(\xi) \in \widetilde{h}^0 (\Th \xi)$ which restricts
to $\pm 1 \in \pi_0 (h)$ on every fiber,
where $1$ denotes the unit of the ring coefficient ring $\pi_* (h)$.
The geometric group $h'_* (X,A)$ of a pair $(X,A)$ of spaces 
is obtained by starting with quadruples of the form
$(M,u(\nu_M),x,f),$ where $M$ is a compact smooth manifold with (possibly empty) boundary, $u(\nu_M)$ is an $h$-orientation of the stable normal bundle $\nu_M$,
$x\in h^* (M)$ is a cohomology class, 
and $f: (M,\partial M)\to (X,A)$ is a continuous map.
Let us briefly write $u_M := u(\nu_M)$.
For $n$-dimensional $M$, the stable Thom class $u_M$ determines a 
homology class
$[M,\partial M]_h \in h_n (M,\partial M)$, the 
\emph{$h_*$-fundamental class} of $(M,\partial M)$,
such that there is a Poincar\'e-Lefschetz duality isomorphism
$h^i (M) \to h_{n-i} (M,\partial M),$ $x\mapsto x\cap [M,\partial M]_h$.
Indeed, assuming for simplicity that $M$ has empty boundary, 
a smooth embedding $M\subset \real^{n+N}$ with normal
bundle $\nu^N$ yields a Thom-Pontrjagin collapse map
$c_N: S^{n+N} \to \Th \nu^N$ to the Thom space of $\nu^N$.
Stably, these $c_N$ induce a morphism of spectra
$c: S^n = \Sigma^{-N} \Sigma^\infty S^{n+N} \to
 \Sigma^{-N} \Sigma^\infty \Th \nu^N = \Th \nu,$ which is sometimes
referred to as the Browder-Novikov morphism.
This morphism determines a class
$[\Th \nu]_h \in h_n (\Th \nu)$ by taking the image of the unit
$1\in h_0 (S)$ under
\[ h_0 (S) \stackrel{\simeq}{\longrightarrow} h_n (S^n)
  \stackrel{c_*}{\longrightarrow} h_n (\Th \nu),
 \]
where the first map is the suspension isomorphism.
Now $[M]_h$ is the image of $[\Th \nu]_h$
under the Thom-Dold isomorphism
$h_n (\Th \nu) \cong h_n (M)$, see Rudyak \cite[p. 319, Prop. 2.8]{Rud:TSOC}.
(Thus the $h^*$-Thom class $u(\nu)$ enters only in the definition of
the Thom-Dold isomorphism.)

Two  quadruples  are called \emph{equivalent} if they differ by a diffeomorphism
that respects the two Thom classes and the two cohomology classes.
Equivalence classes of quadruples are called \emph{cycles}.
The geometric group $h'_* (X,A)$ is the quotient of the free abelian
group generated by these cycles by the following equivalence relations:
\begin{enumerate}
\item $(M, u_M, x+y, f)$ is equivalent to $(M, u_M, x, f) + (M, u_M, y, f),$
\item $(M \sqcup N, u_{M \sqcup N}, x, f)$ is equivalent to 
   $(M, u_M, x|_M, f|_M) + (N, u_N, x|_N, f|_N),$
\item \emph{Bordism}: Two cycles $(M_1, u_{M_1}, x_1, f_1)$ and $(M_2, u_{M_2}, x_2, f_2)$ are bordant if there is a 4-tuple $(W, u_W, y, g)$ such that $W$ is a compact smooth manifold with boundary, $M_1 \sqcup M_2$ is a regularly embedded submanifold of $\pa W,$
\begin{equation*}
	g(\pa W \setminus (M_1 \sqcup M_2)) \subseteq A,
\end{equation*}
the $h^*$-orientations $u_{M_i}$ of $\nu_{M_i}$ are induced by $u_W$ via
\begin{equation*}
	u_{M_i} = (-1)^i u_W|_{M_i},
\end{equation*}
and $y|_{M_i} = x_i,$ $g|_{M_i} = f_i,$
\item \emph{Vector bundle modification}:
  we declare the cycle $(M,u(\nu_M),x,f)$ and  the cycle
  \[ ((S(E\oplus 1), u(\nu_{S(E\oplus 1)}), \sigma_! (x), f\circ \pi)) \] 
   equivalent, where 
 $E\to M$ is an $h^*$-oriented (orthogonal) vector bundle of rank $r$,
 $\pi: S(E\oplus 1)\to M$ denotes the sphere bundle, and
 $\sigma_!: h^* (M) \to h^{*+r} (S(E\oplus 1))$ is the
 cohomological Gysin homomorphism induced by the nowhere zero
 section $\sigma: M \to S(E\oplus 1)$ induced by the map $\sigma (m)=(0_m,1_m)$.
\end{enumerate}

Note that the replacement in a vector bundle modification
may change the dimension of the manifold and it should therefore
be clarified in which sense $h'_*$ is graded.
In fact the group is naturally $\intg$-graded by defining
the homogeneous elements of degree $\K\in \intg$ to be those cycles
$[(M,u(\nu_M),x,f)]$ for which $x|_{M_j}$ is homogeneous and
$\dim M_j - \deg (x|_{M_j}) = \K$ for all $j$, where the $M_j$ are the 
connected components of $M$.
Jakob shows that the natural transformation
\[ \varphi: h'_* (-,-) \longrightarrow h_* (-,-) \]
given by
\[ \varphi [M,u(\nu_M),x,f] := f_* (x \cap [M,\partial M]_h) \]
is an isomorphism on the category of pairs of spaces which have
the homotopy type of finite CW pairs 
(\cite[Theorem 2.3.3.]{Jak:BDH}). \\

For later use in proving Proposition \ref{prop.kappawelldefined}, we record a
base change formula for the cohomological Gysin homomorphism
$\sigma_!$:
\begin{lemma} \label{lem.basechangecohomgysin}
Let
\[ \xymatrix{
E \ar[d] \ar[r]^F & E' \ar[d] \\
M \ar[r]_f & M' 
} \]
be a cartesian diagram, where $f$ is a continuous map between smooth manifolds
and the vertical arrows are the projections of vector bundles.
Suppose that $E'$ is $h$-oriented and $E$ is given the induced orientation.
Let $\sigma': M' \to S' := S(E' \otimes 1)$ and
$\sigma: M \to S := S(E \otimes 1)$ be the sections described above so that the
diagram
\[ \xymatrix{
S  \ar[r]^F & S'  \\
M \ar[u]^\sigma \ar[r]_f & M' \ar[u]_{\sigma'}
} \]
commutes. Then the base change formula
\[ \sigma_! \circ f^* = F^* \circ \sigma'_!: h^* (M') \to h^{*+r} (S) \]
holds.
\end{lemma}
\begin{proof}
We think of $S$ as the union $S= D^+ E \cup_{SE} D^- E$
of two copies of the disc bundle $DE$ glued along their common
boundary $SE$; similarly for $S'$.
Let $u$ and $u'$ be the Thom classes of $E$ and $E'$ in $h$-cohomology.
The two Gysin homomorphisms are given by the vertical compositions
in the following diagram:
\[ \xymatrix{
h^{*+r} (S) & h^{*+r} (S') \ar[l]_{(F\oplus1)^*} \\ 
h^{*+r} (S, D^- E) \ar[u]^{\operatorname{res}} & h^{*+r} (S') 
   \ar[u]_{\operatorname{res}} \ar[l]_{(F \oplus 1)^*} \\
 h^{*+r} (D^+ E, SE) \ar[u]^{\operatorname{exc}}_\cong & h^{*+r} (S') 
   \ar[u]_{\operatorname{exc}}^\cong \ar[l]_{F^*} \\
h^{*} (M) \ar[u]^{-\cup u}_\cong & h^{*} (M') \ar[u]_{-\cup u'}^\cong \ar[l]_{f^*}
} \]
The lower vertical arrows are thus Thom isomorphisms,
the middle vertical arrows are excision isomorphisms,
and the upper vertical arrows are restriction homomorphisms.
The bottom square commutes by naturality of the Thom class,
since the orientation of $E$ is induced by the one of $E'$.
The middle and top square commute, as the vertical maps there
are all induced by inclusions that are compatible with $F$ and $F\oplus 1$.
\end{proof}

\medskip
In this article, we are primarily interested in the two cohomology theories 
$h^* (-) = \K^* ( - )$ and
$h^* (-) = \K^* ( - ) [\tfrac12]$; we denote the  above bordism-like description $h'_* (-)$ of the 
corresponding homology theories as 
$$\K^b_* ( - ) \quad\quad \text{and}\quad\quad \K^{b/2}_* ( - )$$
respectively, where {\bf $b$ stands for {\em bordism}}.

\medskip
We first concentrate on $h^*:= \K^* ( - )[\tfrac12]$,
since this theory will be our main concern. So, let 
$\K^{b/2}_* ( - )$ denote the bordism-like description of $\K_*^\topo ( - )[\tfrac12]$ according to Jakob.
It is important to point out explicitly that {\em $\K^{b/2}_* ( - )$ is  a homology theory}.

Let $(X,A)$ be a pair of topological spaces homotopy equivalent to a finite CW pair. 
Specializing what we have just seen in general, elements in $\K^{b/2}_* (X,A)$ 
are given by equivalences classes of quadruples 
$[(M,\partial M),\Delta_\cplx (\nu_M), x, f:(M,\partial M) \to (X,A)]$ 
with $M$ a smooth compact ($H\intg$-)oriented  manifold, $x\in \K^* (M)[\tfrac12]$ and $\nu_M$ the stable normal bundle of  $M$, which is $\K\smlhf$-oriented by 
$\Delta_\cplx (\nu_M)$, coming from
the Sullivan class of $\nu_M$.  
See Section \ref{sect-siegel-sullivan} above for $\Delta_\cplx (\nu_M)$.
Notice that the notation adopted here is more precise than the one 
in \cite{Jak:BDH}, in that we put the choice of a $\K\smlhf$-orientation  for $\nu_M$, that is $\Delta_\cplx (\nu_M)$,   into the notation.

We consider the free abelian group generated by these cycles and, as already explained,  we first mod out so as to have additivity with respect to
disjoint unions of ($H\intg$-)oriented  manifolds and with respect to sums of classes in 
$\K^* ( - )[\tfrac12]$; we then consider the  equivalence relation on quadruples  
generated by  bordism and vector bundle modification. Regarding the latter we consider  
a real vector bundle $V$ which is ($H\intg$-)oriented; such a bundle is 
$\K[1/2]$-oriented by using $\Delta_\cplx (V)$. As already explained, the  sphere bundle 
$\pi: S:= S(V\oplus 1) \to M$ has  a section $\sigma: M \to S$.
Notice that the sphere bundle $S$ is oriented since $M$ and $V$ are and so possesses a
$\K\smlhf$-orientation.
We  consider $\sigma_! (x)$ with  
$\sigma_! : \K^* (M)[\tfrac12] \to \K^{* + \rk V} (S(V\oplus 1))[\tfrac12]$ 
the Gysin map in cohomology induced by $\sigma$ and declare that 
$$(M,\Delta_\cplx (\nu_M), x, f:(M,\partial M) \to (X,A)) \sim (S(V\oplus 1), 
\Delta_\cplx (\nu_{S(V\oplus 1)}),\sigma_! (x), f\circ \pi: (S(V\oplus 1),\partial) \to (X,A))$$
where we recall that  $S(V\oplus 1)\xrightarrow{\pi} M$ is the natural projection. 
We obtain in this way $\K^{b/2}_* (X)$
and, as already stated, there exists an explicit  isomorphism 
$$\varphi : \K^{b/2}_* (X,A)\to \K^\topo_* (X,A)[\tfrac12]$$
given by
\begin{equation} \label{equ.defjakobvarphikbhalf}
\varphi [M,\Delta_\cplx (\nu_M),x,f] := f_* (x \cap \Delta_\cplx(M,\partial M)). 
\end{equation}

\medskip
Next we consider the generalized cohomology theory given by $h^*:= \K^* $.
A vector bundle is $\K$-\emph{orientable} if and only if it has a 
$\Spin^c$ structure. Hence the homology theory associated 
(by means of a representing spectrum) to the multiplicative cohomology theory
$\K^*_\topo$ may be described using geometric cycles where $M$ is
equipped with a $\Spin^c$ structure.
The $\K^*$-Thom class of $\nu_M$ is then the Atiyah-Bott-Shapiro orientation
$u_{ABS} (\nu_M) \in \widetilde{\K}^0 (\Th \nu_M)$, see \cite[\S 12]{AtiBotSha:CM}.
The corresponding $\K$-homological fundamental class will be written as
$[M,\partial M]_{ABS} \in \K_n (M,\partial M)$.
Altogether, as pointed out by Jakob, $\K^b_* (X)$ has a description
based on cycles of the form $[(M,u_{ABS} (\nu_M),x,f)]$, where $M$ is a compact $\Spin^c$
manifold, $x\in \K^* (M)$, 
$u_{ABS} (\nu_M)$ is the Atiyah-Bott-Shapiro orientation and $f$ is a continuous map from 
$(M,\partial M)$ to $(X,A)$.
According to Bott periodicity,
this group has a natural $\intg/_2$-grading.
The isomorphism 
\[ \varphi: \K^{b}_* (X,A) \longrightarrow \K^\topo_* (X,A) \]
is given by
\[ \varphi [M,u_{ABS} (\nu_M),x,f] := f_* (x \cap [M,\partial M]_{ABS}). \]

\subsection{The homomorphism $\kappa$ to analytic K-homology}
Let $M$ be an oriented compact smooth manifold of dimension $n$.
Let us recall the notation
\begin{equation}\label{notation-sign}
\sign_K (M):= 2^{-\lfloor n/2 \rfloor} [D^{{\rm sign}}_{M}]\in \K_n^\an (M)[\smlhf].
\end{equation}
If $x\in \K^0 (M)[\smlhf]$ then we can twist the signature operator by $x$,
obtaining  an element $[D^{{\rm sign}}_{M,x}]\in \K_*^\an (M)[\smlhf]$; this is done as follows: 
if $x=[E]\otimes \alpha,$ with $\alpha \in \bb Z[\tfrac12],$ then 
$[D^{{\rm sign}}_{M,x}]:= [D^{{\rm sign}}_{M,E}]\otimes \alpha$. 
We define 
\begin{equation}\label{notation-sign-bis}
\sign_K (M,x):= 2^{-\lfloor n/2 \rfloor} [D^{{\rm sign}}_{M,x}]\in \K_*^\an (M)[\smlhf].
\end{equation}

\noindent
We define a group homomorphism  
$$\kappa^{b/2}: \K^{b/2}_*  (X,A) \to \K^\an_* (X,A)[\smlhf]$$ on an element
 $(M^n,\Delta_\cplx (\nu_M), x, f:M\to X)$ as follows:
\begin{enumerate}
\item  if $x\in \K^0 (M)[\smlhf],$ then 
$\kappa^{b/2} ([M^n,\Delta_\cplx (\nu_M), x, f:M\to X]):= 
f_* (\sign_K (M,x))\,;$
\item if $x\in \K^1 (M)[\smlhf],$ we apply a suspension-technique; 
we consider the trivial bundle $V:=M\times \mathbb{R}$ and then $S(V\oplus 1)$, which is nothing but $M\times S^1$, equipped with the first-factor projection $\pi_1: S(V\oplus 1)\to M$; we then take $(S(V\oplus 1), \Delta_\cplx (\nu_{S(V\oplus 1)}), \sigma_! (x), f\circ \pi_1),$
where we observe that by Bott periodicity, the element $\sigma_! (x)\in \K^0 (S(V\oplus 1))[\smlhf]$
can be represented by (a difference of) vector bundles;
we can then set 
\begin{align*}\kappa^{b/2} ([M,\Delta_\cplx (\nu_M), x, f:M\to X])&:= (f\circ \pi_1)_*  (\sign_K (M\times S^1,\sigma_! (x)))\\
&\equiv 2^{-\lfloor (n+1)/2 \rfloor}
 (f\circ \pi_1)_* [D^{{\rm sign}}_{S(V\oplus 1),\sigma_! (x)}]
 \end{align*}
\end{enumerate}

\medskip
\noindent
{\bf Notation.} Unless absolutely necessary, we shall denote the homomorphism $\kappa^{b/2}$ simply by $\kappa$.

\begin{prop} \label{prop.kappawelldefined}
The assignment $\kappa$ 
induces a well-defined homomorphism
\[ \kappa: \K^{b/2}_*  (X,A) \longrightarrow \K_*^\an (X,A)[\smlhf]. \]
\end{prop}
\begin{proof}
The suspension technique has been introduced so as to pass from $\K^1 (-)[\smlhf]$ to $\K^0 (-)[\smlhf]$ 
in a quadruple representing an element in $\K^{b/2}_*  (X)$. Of course we could have achieved
this by considering the bundle $M\times \mathbb{R}^{2k+1}\to M$ or 
in fact any oriented
odd-rank vector bundle $W\to M$, instead
of the trivial line bundle $V\equiv M\times\mathbb{R}$, and then consider $S(W\oplus 1)\xrightarrow{p} M$,
with section $\tau: M\to S(W\oplus 1)$. We claim that 
$$ (f\circ \pi_1)_* \sign_K (S(V\oplus 1),\sigma_! (x))
 = (f\circ p)_* \sign_K (S(W\oplus 1),\tau_! (x))\,,$$
where $x\in \K^1 (M)[\smlhf]$.

Let us briefly write $S$ for $S(W \oplus 1)$.
We choose and fix a base point $a\in S^1$.
Then $\sigma (q) := (q,a)$ defines a map
$\sigma: M \to M\times S^1 = S(V \oplus 1)$, which is a section
for the first factor projection $\pi_1: M\times S^1 \to M$.
The oriented manifold $S\times T^2$, where $T^2 = S^1 \times S^1$ denotes
a $2$-torus, lies over $S$ via the projection
$p_1: S \times T^2 \to S,$ $p_1 (w,t_1,t_2) := w,$
$w\in S,$ $(t_1, t_2) \in S^1 \times S^1 = T^2$.
It also lies over $M \times S^1$ via the projection
$\widetilde{\pi}: S \times T^2 \to M \times S^1$ given by
$\widetilde{\pi} (w,t_1,t_2) := (p(w),t_1)$.
Then the diagram of projections
\begin{equation} \label{dia.projnsst2}
\xymatrix{
M\times S^1 \ar[d]_{\pi_1} & S \times T^2 \ar[l]_{\widetilde{\pi}} 
   \ar[d]^{p_1} \\
M & S \ar[l]^p
} \end{equation}
commutes, as the calculation
\[ \pi_1 \widetilde{\pi} (w, t_1, t_2) 
   = \pi_1 (p(w), t_1) = p(w) = p p_1 (w, t_1, t_2) \]
shows.
The map $s: S \to S\times T^2$ given by
$s(w) := (w,a,a)$ is a section for $p_1$, since
$p_1 s (w) = p_1 (w,a,a) = w$.
The map $\widetilde{\tau}: M \times S^1 \to S \times T^2$
given by $\widetilde{\tau} (q,t) := (\tau (q), t, a)$
is a section for $\widetilde{\pi},$ since
\[ \widetilde{\pi} \widetilde{\tau} (q,t)
 = \widetilde{\pi} (\tau (q), t, a)
 = (p(\tau (q)), t) = (q,t). \]
The diagram of sections
\begin{equation} \label{dia.sectionsst2}
\xymatrix{
M\times S^1 \ar[r]^{\widetilde{\tau}} & S \times T^2  \\
M \ar[u]^{\sigma} \ar[r]_\tau & S \ar[u]_s
} \end{equation}
commutes, as the calculation
\[ \widetilde{\tau} \sigma (q)
  = \widetilde{\tau} (q,a)
  = (\tau (q), a, a)
  = s \tau (q)
 \]
shows. All of these sections are smooth embeddings of oriented
manifolds and thus have associated Gysin restriction homomorphisms.
Now,
\[ p_* \sign_K (S, \tau_! (x))
   =  p_* ([[\tau_! (x)]] \otimes \sign_K (S)), \]
since $\tau_! (x)$ lies in degree $0$ and is thus represented by a 
virtual vector bundle so that we may use (\ref{equ.dsignmeeprime}).
Using the (smooth) Gysin restriction formula (\ref{equ.hilsumgysinrestrofsignk}),
$s^! \sign_K (S \times T^2) = \sign_K (S)$, as $s$ is a smooth oriented
embedding. Consequently,
\[ p_* ([[\tau_! (x)]] \otimes \sign_K (S))
  = p_* p_{1*} s_* ([[\tau_! (x)]] \otimes s^! \sign_K (S \times T^2)). \]
Using (\ref{boardman-analytic-pre}) in Proposition \ref{prop-boardman},
\[ p_* p_{1*} s_* ([[\tau_! (x)]] \otimes s^! \sign_K (S \times T^2))
  = p_* p_{1*} ([[s_! \tau_! (x)]] \otimes \sign_K (S \times T^2)). \]
By the commutativity of diagrams
(\ref{dia.projnsst2}) and (\ref{dia.sectionsst2}),
\begin{align*}
p_* p_{1*} ([[s_! \tau_! (x)]] \otimes \sign_K (S \times T^2)) 
 &= \pi_{1*} \widetilde{\pi}_* ([[\widetilde{\tau}_! \sigma_! (x)]] 
     \otimes \sign_K (S \times T^2)) \\
 &= \pi_{1*} \widetilde{\pi}_* \widetilde{\tau}_*
     ([[\sigma_! (x)]] 
     \otimes \widetilde{\tau}^! \sign_K (S \times T^2)) \\
&= \pi_{1*} ([[\sigma_! (x)]] 
     \otimes \sign_K (M \times S^1)) \\
 &= \pi_{1*} \sign_K (M \times S^1, \sigma_! (x)).            
\end{align*}
Here, we have used that $\sigma_! (x)$ also lies in degree $0$
so that (\ref{equ.dsignmeeprime}) and (\ref{boardman-analytic-pre})
are applicable.
The claim follows by applying $f_*$ to both sides of the equation.

Next we verify that $\kappa$ is compatible with the relation of spherical bundle modification.
Let $[M, \Delta_\cplx (\nu_M), x, f: M\to X]$ be an element of $\K^{b/2}_*  (X)$,
 and let $V\to M$ be an ($H\intg$-) oriented vector bundle, 
 canonically $\K \smlhf$-oriented by $\Delta_\cplx (V)$.
In $\K^{b/2}_*  (X)$, the relation
\[ [M, \Delta_\cplx (\nu_M), x, f] = [S, \Delta_\cplx (\nu_S), \sigma_! (x), f \circ \pi ] \]
holds by definition, and so we must show that the images in $\K^\an_* (X)$
under $\kappa$ of these two representative cycles are equal. There are 4 cases to be considered,
depending on parities of the degree of $x$ and of the rank of $V$.\\
{\bf Case 1:} $x\in \K^0 (M)[\smlhf]$ and $V$ is even rank.\\
In this case, both $x$ and $\sigma_! (x)$ are represented by vector bundles and so
we apply (1) in the above definition of $\kappa$; so, we must show that
\begin{equation}\label{proof-b-mod}
  f_* \sign_K (M,x) 
  =  f_* \pi_* \sign_K (S, \sigma_! (x)). 
 \end{equation}
We have proved the following analytic counterpart of a result of Boardman \cite[p. 35, 6.2.(d)]{boardman},
see Proposition \ref{prop-boardman}:
 \begin{equation}\label{boardman-analytic}[[\sigma_! (x)]] \otimes \sign_K (S)
  = \sigma_* ([[x]] \otimes \sigma^! \sign_K (S)).\end{equation}
  Thus
  \begin{align*}
 \pi_* \sign_K (S, \sigma_! (x))
 &= \pi_* ([[ \sigma_! (x)]] \otimes \sign_K (S)) \\
 &= \pi_*  \sigma_* ([[ x ]] \otimes  \sigma^! \sign_K (S)) \\
 &= [[ x ]] \otimes \sign_K (M) \\
 &=  \sign_K (M,x).
\end{align*}
where we have used  that $\pi \circ \sigma = \id$ and that $ \sigma^! \sign_K (S)=\sign_K (M)$
Thus \eqref{proof-b-mod} holds.\\
{\bf Case 2:} $x\in \K^0 (M)[\smlhf]$ and $V$ is odd rank.\\
In this case, the bundle modification $[S, \Delta_\cplx (\nu_S), \sigma_! (x), f \circ \pi ]$
will have $\sigma_! (x)\in \K^1 (S)[\smlhf]$ and so the definition of $\kappa$ on the right hand side involves suspension, see (2).
Let $\sigma: M \to S=S(V \oplus 1)$ denote the section for $\pi: S(V\oplus 1) \to M$,
and 
let $\tau: S \to S \times S^1$ denote the section for $\pi_1: S \times S^1 \to S$.
In the desired equality
\[ \kappa [M,\Delta,x,f] = \kappa [S,\Delta, \sigma_! (x), f\circ \pi], \]
the left hand side is given by (1) of the definition of $\kappa$,
whereas the right hand side is given by (2).
Hence we must show that
\[ f_* \sign_{K}(M,x) = f_* \pi_* \pi_{1*} \sign_{K}(S\times S^1, \tau_! \sigma_! (x)).  \]
Let $p := \pi \circ \pi_1: S \times S^1 \to M$ denote the composition of the
projections and
let $\omega := \tau \circ \sigma: M \to S \times S^1$ denote the composition
of the sections.
Then
\begin{align*}
f_* \pi_* \pi_{1*}\sign_{K}(S\times S^1, \tau_! \sigma_! (x))
&= f_* p_* ([[\omega_! (x)]] \otimes \sign_{K}(S\times S^1)) \\
&= f_* p_* \omega_* ([[x]] \otimes \omega^! \sign_{K}(S\times S^1)) \\
&= f_* p_* \omega_* ([[x]] \otimes \sign_K(M])) \\
&= f_* ([[x]] \otimes \sign_K (M])) \\
&= f_* \sign_{K}(M,x). 
\end{align*}

\noindent
{\bf Case 3:} $x\in \K^1 (M)[\smlhf]$ and $V$ is even rank.\\
In this case,  $x\in \K^1 (M)[\smlhf]$ and $\sigma_! (x) \in \K^1 (S)[\smlhf]$ and so suspension is needed 
in order to define $\kappa$ on both the left and right hand side of the equation. Since we suspend 
on both sides, the proof reduces easily to Case 1 with $M' = M \times S^1$
and $x' = \tau_! (x) \in \K^0 (M')[\smlhf],$ where $\tau: M \to M\times S^1$ is the appropriate section
for $\pi_1: M\times S^1 \to M$.\\
{\bf Case 4:} $x\in \K^1 (M)[\smlhf]$ and $V$ is odd rank.\\
In this case $\sigma_! (x)$ is in $\K^0 (S)[\smlhf]$ and so while we need to use suspension 
for the left hand side, we do not need to use it for the right hand side. 
Thus we have, using (2) of the definition of $\kappa$,
\[ \kappa [M, \Delta_\cplx (\nu_M), x, f]
   = f_* \pi_{1*} \sign_K (M \times S^1, \tau_! (x)), \]
where $\tau: M \to M \times S^1$ is the section $\tau (q) = (q,a)$
of the factor projection $\pi_1: M\times S^1 \to M$ given by embedding
at a base point $a \in S^1$.
Using branch (1) of the definition of $\kappa$,
\[ \kappa [S, \Delta_\cplx (\nu_S), \sigma_! (x), f \circ \pi ]
   = f_* \pi_* \sign_K (S, \sigma_! (x)), \]
where $\sigma: M \to S$ is the section
of the given projection $\pi: S \to M$ given by using
a nonzero point on the trivial line in $V \oplus 1$.
Therefore, we must prove that
\[ 
   f_* \pi_{1*} \sign_K (M \times S^1, \tau_! (x))
   = f_* \pi_* \sign_K (S, \sigma_! (x)).
\]
This can be done precisely as in the beginning of the proof, 
using the manifold $S \times T^2$
and diagrams of the type 
(\ref{dia.projnsst2}) and (\ref{dia.sectionsst2}).
One should merely note that also in the present context,
both $\sigma_! (x)$ and $\tau_! (x)$ have degree $0$,
so that  (\ref{equ.dsignmeeprime}) and (\ref{boardman-analytic-pre})
are indeed applicable. This finishes case 4, and thus
the proof that $\kappa$ remains invariant under spherical bundle
modification.

Let us now verify that $\kappa$ is invariant under bordism.
Thus, suppose a cycle $(M, \Delta_\cplx (\nu_M), x, f)$ is nullbordant.
For simplicity, let us discuss the case $\partial M = \varnothing$.
Thus there is a nullbordism
 $(W, \Delta_\cplx (\nu_W), y, g)$ such that $M =\partial W$, the $\K\smlhf^*$-orientation of 
 $\nu_M$ is induced by the one on $\nu_W,$
 and $y|_M = x,$ $g|_M = f$. We need to verify that 
 $\kappa (M, \Delta_\cplx (\nu_M), x, f)=0$.
 Let assume first that $x\in \K^0 (M) [\smlhf]$, so that we are in branch (1)
 of the definition.
Let $j: M= \partial W \to W$ denote the inclusion of the boundary,
so that $f = g \circ j$.
Then
\begin{align*}
f_* \sign_K (M,x)
&= f_* ([[x]] \otimes \sign_K (M)) \\
&= g_* j_*  ([[j^* (y)]] \otimes \sign_K (M)) \\
&= g_* ([[y]] \otimes j_* \sign_K (M)) \\
&= g_* ([[y]] \otimes j_* \partial_* \sign_K (W,M)) \\
&= g_* ([[y]] \otimes 0) \\
&= 0.
\end{align*} 
Here, we have used (\ref{equ.dsignmeeprime})  
(which is available since $x$ has degree $0$),
the functoriality Proposition \ref{prop.kaspmultcapnaturality}, 
the connecting homomorphism
$\partial_*: \K_{n+1} (W,M) \to \K_n (M)$ satisfying
$j_* \circ \partial_* =0$, and
 \cite[Thm. 2, p. 49]{RosWei:SO}, see also \cite[Prop. 4.4]{Part1}.
Now suppose that $x\in \K^1 (M) [\smlhf]$, so that branch (2)
of the definition is active. 
We have to show that $f_* \pi_{1*} \sign_K (M \times S^1, \sigma_! (x))$
vanishes.
By the previous argument, it suffices to construct a nullbordism
for $(M\times S^1, \Delta_\cplx, \sigma_! (x), f\circ \pi_1)$.
Such a nullbordism is given by
$(W \times S^1, \Delta_\cplx, \sigma'_! (y), g\circ \pi_1)$, 
where $\sigma': W \to W \times S^1$ is the section given by
embedding at the base point of $S^1$,
since by Lemma \ref{lem.basechangecohomgysin}
\[ (j\times \id)^* (\sigma'_! (y)) =
   \sigma_! (j^* y) = \sigma_! (x). \]
\end{proof}

We end this subsection with a discussion of the analogous map 
$\kappa^b: \K^b_*  (X,A) \to \K^\an_* (X,A)$, with cycles on the left hand side given in terms of 
$\Spin^c$ manifolds.\\
We define a group homomorphism  
$$\kappa := \kappa^b: \K^b_*  (X,A) \to \K^\an_* (X,A)$$ on an element
 $(M^n, u_{ABS} (\nu_M), x, f:M\to X),$ where $M$ is a $\Spin^c$-manifold,
 as follows: 
\begin{enumerate}
\item  if $x\in \K^0 (M),$ then 
$\kappa^b ([M^n,u_{ABS} (\nu_M), x, f:M\to X]):= 
    f_* [D^{\Spin^c}_{M,x}]\,;$ 
\item if $x\in \K^1 (M),$ we again apply a suspension-technique; 
we consider the trivial $\Spin^c$ bundle $V:=M\times \mathbb{R}$ 
and then $S(V\oplus 1)$, which is nothing but the $\Spin^c$ manifold $M\times S^1$, equipped with the first-factor projection $\pi_1: S(V\oplus 1)\to M$; we then take 
 $(S(V\oplus 1), u_{ABS} (\nu_{S(V\oplus 1)}), \sigma_! (x), f\circ \pi_1),$
where we observe that by Bott periodicity, the element $\sigma_! (x)\in \K^0 (S(V\oplus 1))$
can be represented by (a difference of) vector bundles;
we can then set 
$$\kappa^b ([M, u_{ABS} (\nu_M), x, f:M\to X]):= 
 (f\circ \pi_1)_* [D^{\Spin^c}_{S(V\oplus 1),\sigma_! (x)}].$$
\end{enumerate}
\begin{prop}\label{prop:kappawell}
The assignment $\kappa^b$ 
induces a well-defined homomorphism
\[ \kappa^b: \K^b_*  (X,A) \longrightarrow \K_*^\an (X,A). \]
\end{prop}

\noindent
The proof of the proposition is completely analogous to the one given for $\kappa^{b/2}: \K^{b/2}_*  (X,A) \longrightarrow \K_*^\an (X,A)[\tfrac12]$.

\subsection{$\K^0 ( - )$-linearity of the homomorphism $\kappa$}
Recall that we have a cap product 
$$\cap': \K^i (X)\otimes \K^{b/2}_j   (X)\to \K^{b/2}_{j-i}   (X)$$ defined by
$$y\cap' [M, \Delta_{\mathbb{C}} (\nu_M),x,f]:= [M, \Delta_{\mathbb{C}} (\nu_M),f^* y\otimes x,f]\,.$$
See \cite[p. 75, 3.2]{Jak:BDH}.

\begin{lemma}\label{lem.kappa-is-linear}
The homomorphism $\kappa:= \kappa^{b/2}:  \K^{b/2}_0  (X)\to \K^\an_0   (X)[\smlhf]$ is $\K^0 ( - )[\smlhf]$-linear, that is the diagram
\[ \xymatrix{
	\K^0 (X)[\smlhf] \otimes \K^{b/2}_0(X) \ar[d]_{\cap'} \ar[r]^{\id \otimes \kappa} & 
	\K^0 (X)[\smlhf] \otimes \K^\an_0 (X) [\smlhf] \ar[d]^{\otimes} \\
	\K^{b/2}_0(X) \ar[r]_\kappa & 
	\K^\an_0 (X)[\smlhf]
} \]
commutes.
\end{lemma}

\begin{proof}
The naturality property of Proposition \ref{prop.kaspmultcapnaturality} still holds when we invert 2; then we have
\begin{align*}
\kappa ([E] \cap' [M, \Delta_\cplx (\nu_M), x, f])
&= \kappa [M, \Delta_\cplx (\nu_M), [f^* E] \otimes x, f] \\
&= f_* \sign_{K}(M, (f^* E)\otimes x) \\
&= f_* ( [[f^*E]] \otimes \sign_{K}(M,x)) \\
&= [[E]] \otimes f_* \sign_K (M,x) \\
&= [[E]] \otimes \kappa [M, \Delta_\cplx (\nu_M),x,f].
\end{align*}
where Proposition \ref{prop.kaspmultcapnaturality} has been used in the fourth equality.
 \end{proof}
 
 \noindent
 We also have the analogue for $\kappa^b$, proved with the same argument,
 
\begin{lemma}\label{lem.kappa-is-linear}
The homomorphism $\kappa^{b}: \K^{b}_0  (X)\to \K^\an_0   (X)$ is $\K^0 ( - )$-linear, that is the diagram
\[ \xymatrix{
	\K^0 (X) \otimes  \K^{b}_0   (X) \ar[d]_{\cap'} \ar[r]^{\id \otimes \kappa^b} & 
	\K^0 (X) \otimes \K^\an_0 (X) \ar[d]^{\otimes} \\ 
	\K^{b}_0   (X) \ar[r]_{\kappa^b} & 
	\K^\an_0 (X)
} \]
commutes.
\end{lemma}

\subsection{$\kappa$ is an isomorphism}
\begin{thm}
The homomorphism $\kappa^{b/2}: \K^{b/2}_*  ( - )\to \K^\an_*   ( - )[\smlhf]$
induces a natural isomorphism of homology theories.
\end{thm}
\begin{proof}
We shall use the shorter notation $\kappa$ instead of $\kappa^{b/2}$.
That $\kappa$ is a natural transformation of homotopy functors is clear: indeed, if $g:X\to Y$ is a continuous map
between two compact topological spaces, then in $\K^\an_* (Y)$ we have 
\begin{align*}
&g_* (\kappa [M, \Delta_\cplx (\nu_M), x, f])=g_* (f_* \sign_K (M,x))= (g\circ f)_* \sign_K (M,x)\\
& = \kappa [M, \Delta_\cplx (\nu_M), x, g\circ f])= \kappa (g_* [M, \Delta_\cplx (\nu_M), x, f])
\end{align*}
Similarly for a map of pairs. Moreover,   $\kappa$ is a natural transformation of homology theories (i.e. $\kappa$ commutes with the
boundary map); this   follows immediately
from Theorem 2 in \cite{RosWei:SO}; see also Remark 4 there.  Thus in order to prove that $\kappa: \K^b_* (-,-)
\to \K^\an_* (-,-)[1/2]$
is an isomorphism it suffices to consider the topological space equal to a point. We are therefore reduced to show that
$\kappa: \K^b_* ({\rm pt})\to \K^\an_* ({\rm pt})$ is an isomorphism. If $*$ is odd then both groups are trivial and the 
assertion is true. Thus, let  $*$ be  even. Then both groups are equal to $\mathbb{Z}[1/2]$; this is well known 
for $\K^\an_0 ({\rm pt})[1/2]$ and also for $\K^\topo_0 ({\rm pt})[1/2]$; thus by Jakob's isomorphism it is also true
for $\K^b_0 ({\rm pt})$.

Let $R$ be the ring $\K^0 (\pt)[\smlhf]$.
Both $\K^b_0 (\pt)$ and $\K^\an_0 (\pt)[\smlhf]$ are free $R$-modules of rank $1$, as
we shall explain next.
Indeed, by Jakob's Poincar\'e duality theorem \cite[Thm. 4.1]{Jak:cont},
the map
$$\lambda ' : \K^0 (\pt)[\smlhf] \xrightarrow{-\cap' [\pt, \Delta, 1, \id]}  \K^b_0 (\pt)$$
is an isomorphism. Since $\K^0 (\pt)[\smlhf]$ is a free $R$-module generated by the
trivial rank $1$ vector bundle over a point, it follows that 
$\K^b_0 (\pt)$ is a free $R$-module generated by the image of this rank $1$-bundle,
which is $[\pt, \Delta, 1, \id]$.
In order to establish the parallel statement for the analytic $R$-module, we pause
for a moment to discuss the class of the signature operator on a point.
This class, in $\K^\an_0 (\pt)$, is explicitly given by the triple $[D^{{\rm sign}}_\pt] = [H, \alpha, F],$
where $H= H_+ \oplus H_- = \cplx \oplus 0,$
$\alpha$ is scalar multiplication by $\cplx = C(\pt),$
and 
\[ F= \begin{pmatrix}   
0 & F_- \\ F_+ & 0
\end{pmatrix}. \]
Here, $F_+$ is the zero map $\cplx \to 0$
and $F_-$ is the zero map $0 \to \cplx$.
Recall that a generic element of $\K^\an_0 (\pt)$ is given by $[H,\alpha,F]$ where $H=H_+\oplus H_-$,
$\alpha$ is scalar multiplication and $F$ is a $\mathbb{Z}_2$-graded odd operator that is necessarily 
Fredholm. The map $\K^\an_0 (\pt) \ni [H,\alpha,F]\to {\rm ind}(F_+)\in \mathbb{Z}$ is an isomorphism.  
Notice then  that the image of $[D^{{\rm sign}}_\pt]$ under this isomorphism  is equal to $1$.
Tensoring with $\mathbb{Z}[\smlhf]$ one obtains  an isomorphism ${\rm ind}: \K^\an_0 (\pt) [\smlhf]\to R$.
Note that we have a natural homomorphism $\lambda: R\to \K^\an_0 (\pt) [\smlhf]$, 
$\lambda [y]:= [[y]]\otimes [D^{{\rm sign}}_\pt]$. We have also a standard  isomorphism
${\rm rk}: R\to \mathbb{Z}[\smlhf]$ induced by the rank of a bundle.
Then we observe that the following diagram is commutative:
\[ \xymatrix{
R \ar[rr]^\lambda \ar[rd]_\rk^\simeq & & \K^\an_0 (\pt) \ar[ld]^{{\rm ind}}_\simeq \\
& \intg [\smlhf] &
} \]
It follows that $\lambda$ is an isomorphism of $R$-modules.
The diagram
\[ \xymatrix{
& R \ar[ld]_{\lambda'}^\simeq \ar[rd]^\lambda_\simeq & \\
\K^b_0 (\pt) \ar[rr]_\kappa & & \K^\an_0 (\pt)[\smlhf]
} \]
commutes, since
\[ \kappa \lambda' [y]
  = \kappa ([y] \cap' [\pt, \Delta, 1, \id])
   = [[y]] \otimes \kappa [\pt, \Delta, 1, \id]
   = [[y]] \otimes [D^{{\rm sign}}_\pt]
   = \lambda [y]. \]
where we used the $R$-linearity of $\kappa$ provided by Lemma \ref{lem.kappa-is-linear}
in the second equality and the fact that, in this case $[D^{{\rm sign}}_\pt]=\sign_K (\pt)$, given that a point has dimension 0. Consequently $\kappa: \K^b_0 (\pt) \to \K^\an_0 (\pt)[\smlhf]$ is an isomorphism
and the proof of the Theorem is complete.
\end{proof}

\noindent
A similar argument establishes the following result:

\begin{thm}
The homomorphism $\kappa^{b}: \K^{b}_*  ( - )\to \K^\an_*   ( - )$
 induces a natural isomorphism of homology theories.
\end{thm}

\subsection{Compatible $\K^* ( -)$-module structures on
$\K^\topo_* ( - )$, $\K^{b/2}_* ( - )$ and $\K^{{\rm b}}_* ( - )$.}

\medskip
\noindent
Recall  the cap product $\cap'$ on  $\K^{b/2}$-homology, given by
\[ \K^* (X) \otimes \K^{b/2}_i (X) \longrightarrow \K^{b/2}_{i-*} (X), \]
\[ y \cap' [M,\Delta_\cplx (\nu_M),x,f] := [M, \Delta_\cplx (\nu_M),f^* (y) \otimes x, f], \]

\noindent
Recall also that for any compact topological space $X$, Jakob has defined an isomorphism 
\begin{equation}\label{jacob}
\varphi: \K^{b/2}_{\ell} (X) \rightarrow \K^\topo_{\ell} (X)[\smlhf]
\end{equation} 
that satisfies the following compatibility relation
\begin{equation}\label{jakob-compatibility} 
 \xymatrix{
	\K^* (X) \otimes \K^{b/2}_j (X) \ar[d]_-{\id \otimes \varphi} \ar[r]^-{\cap'} & 
	\K_{j-*}^{b/2}  (X) \ar[d]^-{\varphi} \\
	\K^* (X) \otimes \K^\topo_{j} (X)[\smlhf] \ar[r]^-{\cap} & 
	\K^\topo_{j-*} (X)[\smlhf]
} 
\end{equation}
See again \cite{Jak:BDH}, page 75, Section 3.2.

If $N$ is a closed oriented manifold
of dimension $n$, then we 
have fundamental classes
\begin{equation}\label{equ.fundamental-oriented}
  \Delta_\cplx (N) \in \K^\topo_n (N)[\smlhf] \quad\text{and}\quad 
  [N, \Delta_\cplx (\nu_N), 1, \id_N] \in \K^{b/2}_n (N)
\end{equation}
and  Poincar\'e duality 
isomorphisms
\begin{equation}\label{equ.poincare-oriented}
	\K^*(N) \xrightarrow{-\cap \Delta_\cplx (N)} 
	\K^\topo_{n-*} (N) 
	\quad\text{and}\quad 
	\K^*(N) \xrightarrow{-\cap' \, [N, \Delta_\cplx (\nu_N),1,\id_N]} 
	\K^b_{n-*} (N)\,.
\end{equation}
The next Lemma shows that Jakob's isomorphism is compatible with the two Poincar\'e 
duality isomorphisms. 
\begin{lemma} \label{lem.orientedpddeltafundjakobcommwphi}
Let $N$ be a closed oriented manifold of dimension $n$.
Then the diagram
\[ \xymatrix{
& \K^* (N)[\smlhf] 
\ar[ld]_{-\cap \Delta_\cplx (N)}^\cong 
   \ar[rd]^{~ -\cap' \,
   [N, \Delta_\cplx (\nu_N), 1,\id_N]}_\cong & \\
\K^\topo_{n-*} (N)[\smlhf] 
& & \K^{b/2}_{n-*} (N)
 \ar[ll]_{\varphi}^\cong 
} \]
commutes. A similar result holds for compact manifolds with boundary.
\end{lemma}
\begin{proof}
Let $x$ be any element in $\K^* (N)[\smlhf]$. Then by the compatibility relation
(\ref{jakob-compatibility}),
\[ \varphi (x \cap' [N, \Delta_\cplx (\nu_N) ,1,\id_N])
   = x \cap \varphi  [N, \Delta_\cplx (\nu_N) ,1,\id_N] \]
By (\ref{equ.defjakobvarphikbhalf}),
$\varphi  [N, \Delta_\cplx (\nu_N) ,1,\id_N]
  = (\id_N)_* (1 \cap \Delta_\cplx (N)) = \Delta_\cplx (N)$ and the proof is complete.
\end{proof}

\noindent
There is an analogous isomorphism
\begin{equation}\label{jacob-bis}
\varphi: \K^{b}_{\ell} (X) \rightarrow \K^\topo_{\ell} (X)
\end{equation} 
that satisfies 
\begin{equation}\label{jakob-compatibility-bis} 
 \xymatrix{
	\K^* (X) \otimes \K^{b}_j (X) \ar[d]_{\;\;\;\id \otimes \varphi} \ar[r]^-{\cap'} & 
	\K_{j-*}^{b}  (X) \ar[d]^{\varphi} \\
	\K^* (X) \otimes \K^\topo_{j} (X) \ar[r]^-{\cap} & 
	\K^\topo_{j-*} (X)
} 
\end{equation}

\medskip
\noindent
Similarly, if $N$ is a closed $\Spin^c$-manifold 
of dimension $n$, then we 
have fundamental classes
\begin{equation}\label{fundamental-spin_c}
[N]_K\in \K^\topo_n (N) \quad\text{and}\quad [N,1,\id_N]\in \K^b_n (N)
\end{equation}
and  Poincar\'e duality 
isomorphisms
\begin{equation}\label{poincare-k-spin_c}
	\K^*(N)\xrightarrow{-\cap [N]_K} 
	\K^\topo_{n-*} (N)\quad\text{and}\quad 
	\K^*(N) \xrightarrow{-\cap' \, [N,1,\id_N]} \K^b_{n-*} (N)\,.
\end{equation}
Then we have the result analogous to Lemma \ref{lem.orientedpddeltafundjakobcommwphi}, whose proof we omit.
\begin{lemma} \label{lem.pddeltafundjakobcommwphihphi}
Let $N$ be a closed $\Spin^c$-manifold of dimension $n$.
Then the diagram
\[ \xymatrix{
& \K^* (N)
\ar[ld]_{-\cap [N]_K}^\cong 
   \ar[rd]^{-\cap' \,
   [N,1,\id_N]}_\cong & \\
\K^\topo_{n-*} (N)
& & \K^b_{n-*} (N)
 \ar[ll]_{\varphi} 
} \]
commutes. A similar result holds for compact $\Spin^c$-manifolds with boundary.
\end{lemma}

\subsection{Comparison with the Baum-Douglas K-homology group}
We can combine the isomorphism $\kappa$ with the isomorphism between the geometric K-homology groups of Baum-Douglas and the analytic K-homology groups established in \cite{baumhigsonschick},
\begin{equation*}
	\mu: \K^{\geo}_*(X) \lra \K^{\an}_*(X)
\end{equation*}
to obtain an isomorphism
\begin{equation*}
	\kappa^{-1} \circ \mu: \K^{\geo}_*(X) \lra \K^{b}_*(X)
\end{equation*}
between the geometric K-homology groups and the K-homology groups described by Jakob in \cite{Jak:BDH}\footnote{This clarifies the equality of these two homology theories, already pointed out in {\em loc cit}.}.
This isomorphism has a simple description in terms of representative cycles. Recall, e.g.,  that a class in $\K^{\geo}_0(X)$ can be represented by a triple $(M,E,f)$ 
wherein $M$ is a $\Spin^c$ manifold, $E$ is a complex vector bundle, and 
$f:M \lra X$ is a continuous map. If $[E]$ denotes the class of $E$ in 
$\K^0(X)$ then $(M,u_{ABS}(\nu_M), [E], f)$ represents an element of 
$\K^{b}_0(X)$ and $\kappa^{-1}\circ \mu$ is induced by the map
\begin{equation*}
	(M,E,f) \mapsto (M,u_{ABS}(\nu_M), [E], f).
\end{equation*}
Indeed, we have that
\begin{equation*}
	\mu( M,E,f) = f_*[D_{E}] = \kappa(M,u_{ABS}(\nu_M), [E], f)
\end{equation*}
where $[D_E]$ denotes the $\K^{\an}_*$-class of the $\Spin^c$ operator on $M,$ twisted by (some choice of connection on) $E.$

\section{Compatibility of Gysin maps for oriented inclusions and the signature operator}\label{sec:CompHoriented}
 
Using the isomorphism $\kappa$ from the previous section we may compare the $\KO^\topo_*[\smlhf]$ and $\K^\an_*[\smlhf]$ groups of a CW-complex $B$ by the map
\begin{equation}\label{eq:def-of-lambda}
	\lambda: \KO^\topo_n (B)[\smlhf] \lra \K^\an_n (B)[\smlhf] 
\end{equation}
defined as the composition
\begin{equation*}
	\xymatrix{
	\KO^\topo_n (B)[\smlhf]  \ar[r]^-{(\Psi^2)^{-1}} &
	\KO^\topo_n (B)[\smlhf] \ar[r]^-{c} &
	\K^\topo_n (B)[\smlhf]  \ar[r]^-{\varphi^{-1}} &
	\K^{b/2}_n (B) \ar[r]^-{\kappa} &
	\K^\an_n (B)[\smlhf] 
	}
\end{equation*}
where $\Psi^2$ is the second Adams operation, $c$ is complexification, and $\varphi$ is the map from \eqref{jacob-bis}.

Our aim in this section is to establish that $\lambda$ intertwines the Gysin maps induced by normally nonsingular inclusions, so we start by recalling the definition of these maps.

\begin{defn} \cite[p.46]{GorMac:SMT} \label{def:nnsi}
An inclusion  $j:B \hookrightarrow Y$
of locally compact Hausdorff spaces is said to be 
{\bf normally nonsingular}\footnote{In Part 1, when discussing normally nonsingular inclusions of stratified spaces, we had the extra requirement that $\phi$ was a stratified diffeomorphism and that the normal bundle was oriented.} 
if there is a vector bundle $\nu$ over $B,$ with total space $X,$ 
a neighborhood $U$ of $B$ in $Y,$ and a homeomorphism $\phi: X \to U$ such that 
the diagram
\begin{equation*}
	\xymatrix{
	X \ar[r]^-{\phi} & U  \\
	B \ar@{^(->}[u] \ar[r]^j_\simeq & j(B), \ar@{^(->}[u]
} \end{equation*}
commutes, where the left hand vertical map is the zero section inclusion.
The bundle $\nu$ is referred to as the normal bundle of the inclusion and $U$ is referred to as a tubular neighborhood of $B$ in $Y.$
\end{defn}

The desired result is thus:

\begin{thm}\label{thm.compatibility-immersions}
Let $B,Y$ be compact topological spaces having the homotopy type of a finite CW-complex.
Let $g:B\hookrightarrow Y$ be a normally nonsingular inclusion 
with oriented normal bundle $\nu$ of rank $r$. Then the diagram 
\begin{equation}\label{9}
 \xymatrix{
\KO^\topo_n (Y)[\smlhf]  \ar[d]_-{g^!_\topo} \ar[r]^-\lambda
 & \K^\an_n (Y)[\smlhf] \ar[d]^-{g^!_\an} \\
\KO^\topo _{n-r} (B)[\smlhf]  \ar[r]^-\lambda
 & \K^\an_{n-r} (B)[\smlhf] 
} \end{equation}
commutes.
\end{thm}

The rest of this section is devoted to the proof of this result. 

To start with, note, using the notation from Definition \ref{def:nnsi} and denoting the inclusion of $B$ into $D\nu$ by $j,$ that both $g^!_\topo$ and $g^!_\an$ are defined by the composition of 3 homomorphisms; the first
is induced by the restriction to $U$, the second one is $\phi_*$ and the third is given by
$$ \KO^\topo_n (D \nu,S\nu)[\smlhf] \xrightarrow{\Delta_{{\rm SO}} (\nu)\cap -} \KO^\topo_{n-r} (B)[\smlhf]
\quad
\text{and}
\quad
\K^\an_n (D \nu,S\nu)[\smlhf] \xrightarrow{j^!_\an} \K^\an_{n-r} (B)[\smlhf]
$$
respectively.
The part of the diagram corresponding to the first two maps commutes by functoriality.
The main step is therefore the proof of the commutativity of the following diagram, where we abbreviate
\begin{equation*}
	\gamma = c\circ (\Psi^2)^{-1},
\end{equation*}
\begin{equation}\label{commutativity-immersions-main} 
\xymatrix{
	\KO^\topo_n (D \nu,S\nu)[\smlhf] \ar[d]_-{\Delta_{{\rm SO}} (\nu)\cap -} \ar[r]^-\gamma & 
	\K^\topo_n (D\nu,S\nu)[\smlhf] \ar[d]_-{\Delta_{\mathbb{C}} (\nu)\cap -} & 
	\K^{b/2}_n (D\nu,S\nu) \ar[l]_-{\varphi '} \ar[d]_-{\Delta_{\mathbb{C}} (\nu)\cap' -} \ar[r]^-\kappa & 
	\KK_n (C_0 X,\cplx)[\smlhf] \ar[d]^-{j^!_\an} \\ 
	\KO^\topo_{n-r} (B)[\smlhf] \ar[r]_-\gamma & 
	\K^\topo_{n-r} (B)[\smlhf]  & 
	\K^{b/2}_{n-r} (B) \ar[l]^-{\varphi '}\ar[r]_-\kappa & 
	\KK_{n-r} (C B,\cplx)[\smlhf] 
} 
\end{equation}

\begin{prop}
The first square on the left of \eqref{commutativity-immersions-main} commutes.
\end{prop} 
\begin{proof}
Both the stable second Adams operation $\Psi^2: \KO \smlhf \to \KO \smlhf$
and the complexification $c: \KO \smlhf \to \K \smlhf$ are multiplicative morphisms
of ring spectra. Thus their composition $\gamma: \KO \smlhf \to \K \smlhf$
is multiplicative. It follows that the diagram
\[ \xymatrix{
\KO^* (X)[\smlhf] \otimes \KO_* (X)[\smlhf] \ar[d]_{\cap_\KO} 
   \ar[r]^-{\gamma \otimes \gamma } &
  \K^* (X)[\smlhf] \otimes \K_* (X)[\smlhf] \ar[d]^{\cap_\K} \\
  \KO_* (X)[\smlhf] \ar[r]_\gamma & \K_* (X)[\smlhf] 
} \]
commutes. Hence, fixing the cohomological variable to be the Sullivan orientation,
\[ \xymatrix{
\KO_* (X)[\smlhf] \ar[d]_{\Delta_{SO} (\nu)\cap -} \ar[r]^-{\gamma \otimes \gamma } &
  \K_* (X)[\smlhf] \ar[d]^{\gamma \Delta_{SO} (\nu) \cap -} \\
  \KO_* (X)[\smlhf] \ar[r]_\gamma & \K_* (X)[\smlhf] 
} \]
commutes. By definition, $\gamma \Delta_{SO} (\nu) = \Delta_{\cplx} (\nu)$.\\
\end{proof}

\medskip
\noindent
The commutativity of the central square of \eqref{commutativity-immersions-main} is shown in \cite{Jak:BDH}.

\medskip
\noindent
Thus to prove Theorem \ref{thm.compatibility-immersions} we only need to show the commutativity of the square on the right in \eqref{commutativity-immersions-main}. To carry this out, it will be convenient to make use of transversality results to define a map
\begin{equation*}
	j^!_{\pitchfork}: \K^{b/2}_n (D\nu,S\nu) \lra \K^{b/2}_{n-r} (B)
\end{equation*}
which we will show is a different but equivalent description of the Gysin map in $\K^{b/2}_*.$ We start by recalling the notion of transversality relevant in this context from \cite[\S II.2]{browder}.
\begin{defn}
Suppose $B$ is a finite complex, $B \hookrightarrow Y$ is a normally nonsingular 
inclusion, and $f:M \to Y$ is a continuous map from a smooth manifold $M.$ 
Let $r$ be the rank of the normal bundle of $B$ in $Y$ and let $n$ denote the 
dimension of $M.$ We say that $f$ is {\bf transverse} to $B$ if 
$f^{-1}(B) =: N$ is a smooth submanifold of $M$ (of dimension $n-r$) 
with normal bundle $\nu'$ and $f$ restricted to a 
sufficiently small
tubular neighborhood of $N$ in $M$ is a linear bundle map of $\nu'$ into $\nu,$
which is an isomorphism on every fiber.
\end{defn}

Thom's transversality theorem (see Browder \cite[p. 34, Thm. II.2.1]{browder}) says, with the notation as above, that: 
\begin{quote}
Let $A$ be closed in $M$ and suppose that $f$ restricted to an open neighborhood of $A$ is already transverse to $B.$ Then there is a homotopy of $f \text{ rel }A$ to 
$\widetilde f$ such that $\widetilde f$ is transverse to $B.$
\end{quote}
Using this theorem we may, starting with an element $[M, \Delta_{\cplx}(\nu_M), x, f:M\to D\nu]\in \K^{b/2}_* (D\nu,S\nu)$
with $x\in \K^* (M)[\smlhf],$ replace $f$ with a homotopic map $\widetilde{f}: (M,\partial M)\to (D\nu, S\nu)$ such that $\widetilde{f}$ is transverse to the zero-section $B$ in $D\nu$. In particular, $N = \widetilde f^{-1}(B)$ is a smooth submanifold of $M^\circ$ and we obtain a vector bundle isomorphism
\begin{equation} \label{equ.nunisognu}
\nu_N \cong g^* \nu, 
\end{equation} 
where $g:= \widetilde{f}|: N\to B$. We will write $V_N$ for the total space of $\nu_N$. The smooth manifold $N$ is compact and has empty boundary. It is oriented, since both $M$ and $\nu_N$ are oriented. We denote by $\iota:N\to M$ (that is, iota) the inclusion.

\begin{defn}\label{def:TransverseGysin}
With the notation as above, to each  $[M, \Delta_{\cplx}(\nu_M), x, f:M\to D\nu] \in \K^{b/2}_i (D\nu, S\nu) $ we associate
$$j^!_{\pitchfork} [M, \Delta_{\cplx}(\nu_M), x, f:M\to D\nu]
  := [N, \Delta_{\cplx}(\nu_N), \iota^* (x), g:N\to B]$$
\end{defn}

\begin{prop}\label{prop-fork-welldefined}
The assignment
$ j^!_{\pitchfork}$ induces a well-defined homomorphism $j^!_{\pitchfork}: \K^{b/2}_n (D\nu,S\nu)
\to \K^{b/2}_{n-r} (B)\,.$
\end{prop}

\begin{proof}
To check that the bordism relation on representatives leaves the map invariant we proceed as follows.
Suppose that $(M, \Delta_{\cplx}(\nu_M), x, f:M\to D\nu)$ and $(M', \Delta_{\cplx}(\nu_{M'}), x', f:M'\to D\nu)$ are cycles representing the same class in $\K^{b/2}_*(D\nu, S\nu)$ (in particular $f(\pa M)$ and $f'(\pa M')$ are subsets of $S\nu$) that are related by a bordism. That is, suppose that there is a manifold with bounday $W$ with $\pa W = M \cup M' \cup \omega$ (so $\omega$ is a bordism between $\pa M$ and $\pa M'$) and a 4-tuple $(W, \Delta_{\cplx}(\nu_W), x_W, F:W \to D\nu)$ whose restriction to $M,$ respectively $M',$ is equal to $(M, \Delta_{\cplx}(\nu_M), x, f:M\to D\nu),$ respectively $(M', \Delta_{\cplx}(\nu_{M'}), x', f:M'\to D\nu),$ and which satisfies
\begin{equation*}
	F(\omega) \subseteq S\nu.
\end{equation*}
Using Thom's transversality theorem, we may and do assume that $F,$ $f,$ and $f',$ are transverse to the zero section $B$ of $D\nu.$ Now we note that $F^{-1}(B)$ is a manifold with boundary which is a bordism between $f^{-1}(B)$ and $(f')^{-1}(B).$ Thus, with $N= f^{-1}(B),$ $g = f|_N$ and $N' = (f')^{-1}(B),$ $g' = f'|_{N'},$ we see that $(N, \Delta_{\cplx}(\nu_N), \iota^* (x), g:N\to B)$ and $(N', \Delta_{\cplx}(\nu_{N'}), (\iota')^* (x), g':N'\to B)$ represent the same element of $\K^{b/2}_{n-r} (B)$ as desired.

Next let us address the effect of bundle modification.
Consider $[M, \Delta_{\cplx}(\nu_M), x, f:M\to D\nu]\in \K^{b/2}_* (D\nu,S\nu)$ and let $V\to M$ be an orientable vector bundle; consider
$S:= S(V\oplus 1)\xrightarrow{\pi} M$.
By definition of bundle modification,  \[ (M, \Delta_{\cplx}(\nu_M), x, f:M\to D\nu) \sim (S,  \Delta_\cplx (\nu_S), \sigma_! (x), f \circ \pi: S\to D\nu). \]
Using the same notation as above Definition \ref{def:TransverseGysin}, we observe that $\widetilde{f}\circ \pi\simeq f\circ \pi$. Since $\pi$ is a smoothly locally trivial projection over a smooth manifold (namely $M$),
it is transverse to any submanifold of  $M$.
Hence 
$\widetilde{f}\circ \pi$ is transverse to $B$ in $D\nu$. Consider then 
$$N_S:= (\widetilde{f}\circ \pi)^{-1} (B)\,;\quad N_S \xhookrightarrow{\,\iota_S} S\quad\text{and}\quad g_S:= \widetilde{f}\circ \pi |_{N_S}: N_S\to B\,.$$
Observe that $N_S=S(V|_N\oplus 1)$; let $\sigma_N: N\to S(V|_N\oplus 1)$ be the section induced by $\sigma_N: N\to V|_N\oplus 1$,  $\sigma_N (n)= (0,1_n)$; then the following diagram is commutative

\[ \xymatrix{
N_S \ar[r]^-{\iota_S} \ar[d] & S(V\oplus 1) \ar[d] \\
N \ar@/^1pc/[u]^{\sigma_N} \ar[r]_\iota & M \ar@/_1pc/[u]_\sigma
} \]
We then have
\begin{align*}
j^!_{\pitchfork}[ S,  \Delta_\cplx (\nu_S), \sigma_! (x), f \circ \pi: S\to D\nu]&= [N_S, \Delta_\cplx (\nu_{N_S}), \iota^*_S \sigma_! (x), 
g_S: N_S\to B]\\&= [S(V|_N\oplus 1), \Delta_\cplx (\nu_{S(V|_N\oplus 1)}), (\sigma_N)_! \iota^* (x),g_S: S(V|_N\oplus 1)\to B]
\end{align*}
where the equality $ (\sigma_N)_! \iota^* (x)=\iota^*_S \sigma_! (x)$ can be justified by adopting 
the definition of  Gysin map in $\K^*$ as in \cite[p. 75]{Jak:BDH}, and then using the naturality of the Thom class and
the fact that $\iota^*$ is multiplicative.
We recognize in the last term the bundle modification of $[N,\Delta_\cplx (\nu_N), \iota^* (x), g:N\to B]$ and since the latter
is, by definition,  $j^!_{\pitchfork} [M, \Delta_{\cplx}(\nu_M), x, f:M\to D\nu]$ we conclude that
\[ j^!_{\pitchfork}[ S,  \Delta_\cplx (\nu_S), \sigma_! (x), f \circ \pi: S\to D\nu]= j^!_{\pitchfork} [M, \Delta_{\cplx}(\nu_M), x, f:M\to D\nu]\]
which is precisely what we wanted to show.
\end{proof}

\bigskip
\noindent
 We then have the crucial

\begin{prop}\label{lemma:gysin-geo-trans-bis}
For any $\alpha\in \K^{b/2}_* (D\nu,S\nu)$ we have 
\begin{equation}\label{pitch'} j^!_{b/2} \alpha=  j^!_{\pitchfork} \alpha
\end{equation}
with $$j^!_{b/2} \alpha:= \Delta_{\mathbb{C}} (\nu)\cap' \alpha$$
\end{prop}
\begin{proof}
First we make the remark that if $N$ is orientable smooth and $j_N:N\hookrightarrow D\nu$ is the inclusion
of $N$ as the zero section of an orientable bundle $\nu$ of rank $r$ then 
$j^!_{b/2} $, as a map from $\K^{b/2}_i (D\nu,S\nu)$ to  $\K^{b/2}_{i-r} (N)$, is given by  
\begin{equation}\label{explicit-board} 
\xymatrix@C=60pt{
	\K^{b/2}_i (D\nu,S\nu) & 
	\K^{b+r-i} (D\nu)[\smlhf] \ar[l]_{-\cap' [(D\nu,S\nu),1,\id]} \ar[r]^{j_N^*} & 
	\K^{b+r-i} (N)[\smlhf]  \ar[r]^{-\cap' [N,1,\id]} & 
	\K^{b/2}_{i-r} (N)
} \end{equation}
where the first and the last map are Poincar\'e duality isomorphisms that are defined by $\cap'$-product 
with the fundamental classes $[(D\nu,S\nu),1,\id]\in \K^{b/2}_{n+r}(D\nu,S\nu)$ and $[N,1,\id]\in \K^{b/2}_n (N)$ respectively. Here, for the sake of simplicity, we did not write the Sullivan orientation
into the notation of the fundamental classes. See \cite[Theorem 6.24]{boardman}.

 Let $[M, \Delta_{\cplx}(\nu_M), x, f:M\to D\nu]
  \in \K^{b/2}_* (D\nu,S\nu)$ where, without loss of generality, we can assume that $f$ is transverse 
to $B$ viewed as the zero-section of $\nu$. Our goal is to prove that 
\begin{equation}\label{equality-gysin}
j^!_{b/2}   [M, \Delta_{\cplx}(\nu_M), x, f:M\to D\nu]= j^!_{\pitchfork} [M, \Delta_{\cplx}(\nu_M), x, f:M\to D\nu]\,.
\end{equation}
We shall use the fact that, by definition,  $[M, \Delta_{\cplx}(\nu_M), x, f:M\to D\nu]=f_* [M, \Delta_{\cplx}(\nu_M), x, 
{\rm id}_M:M\to M]$.
Consider $N:=f^{-1} (B)$, an orientable manifold contained in the interior of $M$. Then $N$ admits 
a normal bundle, denoted $\nu'$; in fact $\nu'= (f|_N)^* \nu$.
We identify  $D\nu'$ with a tubular neighbourhood of $N$ in $M$ without making this identification explicit in our notation. We denote by iota, $\iota$,  the inclusion of $N$ into $M$ and observe that $\iota= i \circ j_N$ with $j_N: N\to D\nu'$ the inclusion as the zero-section and $i: D\nu' \to M$ the natural  
inclusion. 
Observe that there are two homomorphisms
$$\iota^!_{b/2}: \K^{b/2}_* (M,\partial M)\to \K^{b/2}_{*-r} (N),
  \quad \quad 
\iota^!_{\pitchfork}: \K^{b/2}_* (M,\partial M)\to \K^{b/2}_{*-r} (N)\,;$$
these are defined as
$$ \iota^!_{b/2}:=  (j_N)^!_{{b/2}} \circ i^!, 
  \quad\quad 
\iota^!_{\pitchfork}:=  (j_N)^!_{\pitchfork} \circ i^!$$
with $i^!: \K^{b/2}_* (M,\partial M)\to \K^{b/2}_* (D\nu' ,S\nu')$ the Gysin map induced by the codimension 0 inclusion $i: D\nu' \to M$. Notice that $i^!: \K^{b/2}_* (M,\partial M)\to \K^{b/2}_* (D\nu' ,S\nu')$ is the composition of the restriction homomorphism
$\K^{b/2}_* (M,\partial M)\to \K^{b/2}_* (M,(M\setminus N)\cup \partial M)$
followed by the inverse of the excision isomorphism
$${\rm e}: \K^{b/2}_* (D\nu',S\nu') \to 
  \K^{b/2}_* (M, (M\setminus N)\cup \partial M).$$ 
We shall need the following two lemmas.

\begin{lemma} \label{lem.basechangepitchfork-oriented}
(Base change for $j^!_\pitchfork$.)\\
Let $g:B' \to B$ be a continuous map between finite CW-complexes,
covered by a vector bundle map
\[ \xymatrix{
V' \ar[d] \ar[r]^G & V \ar[d] \\
B' \ar[r]_g & B
} \]
between oriented vector bundles,
which is oriented and a fiberwise isomorphism.
Then the diagram
\[ \xymatrix{
	\K^{b/2}_* (DV',SV') \ar[d]_{j^!_\pitchfork} \ar[r]^{G_*} & 
	\K_*^{b/2} (DV,SV) \ar[d]^{j^!_\pitchfork} \\
	\K^{b/2}_{*-r} (B') \ar[r]_{g_*} & 
	\K^{b/2}_{*-r} (B)
} \]
commutes.
\end{lemma}
\begin{proof}
Let $[M,\Delta_{\cplx}(\nu_M),x,f] \in \K^{b/2}_*(DV', SV')$ be any element.
Then
$$G_* [M,\Delta_{\cplx}(\nu_M),x,f] = [M,\Delta_{\cplx}(\nu_M),x, G\circ f].$$
Let $\widetilde{f}: (M,\partial M) \to (DV',SV')$ be a map which is 
homotopic to $f$ and transverse to $B'$ in $DV'$.
Let $N := \widetilde{f}^{-1} (B') \subset M$ denote the transverse
preimage of the zero section.
Thus 
$$
j^!_{\pitchfork} [M, \Delta_{\cplx}(\nu_M), x, f:M\to DV']= [N, \Delta_{\cplx}(\nu_N), \iota^* (x),  \widetilde{f}|_N: N\to B']$$
and consequently,
$g_* j^!_\pitchfork [M, \Delta_{\cplx}(\nu_M), x, f:M\to DV'] = [N, \Delta_{\cplx}(\nu_N), \iota^* (x), g \circ \widetilde{f}|_N]$.
We claim that $G\circ \widetilde{f}$ is transverse to $B$ in $DV$.
Indeed,
since $G$ is a vector bundle map and a fiberwise isomorphism,
the preimage 
\[ 
(G\circ \widetilde{f})^{-1} (B) = \widetilde{f}^{-1} (G^{-1} (B)) 
 =  \widetilde{f}^{-1} (B') 
\]
is precisely $N$, and $G\circ \widetilde{f}$ is a linear bundle
map and a fiberwise isomorphism from the normal bundle of $N$
in $M$ to the normal bundle of $B$ in $DV$, which is $V$.
Note also that $G \circ \widetilde{f}$ is homotopic to $G\circ f$.
Therefore,
\begin{align*}
j^!_\pitchfork G_*  [M, \Delta_{\cplx}(\nu_M), x, f:M\to DV'] &= j^!_\pitchfork 
[M, \Delta_{\cplx}(\nu_M), x, G\circ f:M\to DV]\\
   &= j^!_\pitchfork [ M, \Delta_{\cplx}(\nu_M), x, G\circ \widetilde{f}:M\to DV]\\
  &= [N,\Delta_{\cplx}(\nu_N), \iota^* x, g \circ \widetilde{f}|_N:N\to B].
\end{align*}
and the latter is precisely $g_* j^!_\pitchfork [M, \Delta_{\cplx}(\nu_M), x, f:M\to DV'].$
\end{proof}

\begin{lemma} \label{lem.basechangeThom-oriented}
(Base change for $j^!_{b/2}$.) 
Under the assumptions of Lemma \ref{lem.basechangepitchfork-oriented},
 the diagram
\[ \xymatrix{
	\K^{b/2}_* (DV',SV') \ar[d]_{j^!_{b/2}} \ar[r]^{G_*} & 
	\K_*^{b/2} (DV,SV) \ar[d]^{j^!_{b/2}} \\
	\K^{b/2}_{*-r} (B') \ar[r]_{g_*} & 
	\K^{b/2}_{*-r} (B)
} \]
commutes.
\end{lemma}
\begin{proof}
Recall that $j^!_{b/2}$ is defined as the composition of $\K^{b/2}_* (DV,SV)\xrightarrow{\Delta_\cplx (V)\cap' -} \K^{b/2}_{*-r} (V)$ and the inverse of the isomorphism $j_*: \K_{*-r}(B)\to \K_{*-r} (V)$, with $j: B\hookrightarrow V$
the inclusion of the zero-section. By functoriality we only need
to show the commutativity of 
\[ \xymatrix{
	\K^{b/2}_* (DV',SV') \ar[d]_{\Delta_\cplx (V')\cap' -} \ar[r]^{G_*} & 
	\K_*^{b/2} *(DV,SV) \ar[d]^{\Delta_\cplx (V)\cap' -} \\
	\K^{b/2}_{*-r} (DV') \ar[r]_{G_*} & 
	\K^{b/2}_{*-r} (DV)
} \]
By naturality of the Sullivan classes we know that $G^* \Delta_\cplx (V)= \Delta_\cplx (V')$, given that $G$ is orientation preserving. Thus for any $x\in \K^{b/2}_* (DV',SV')$ we have 
$$G_* (\Delta_\cplx (V')\cap' x)= G_* (G^* \Delta_\cplx (V)\cap' x)= \Delta_\cplx (V)\cap' G_* (x)$$
and the proof is complete.
\end{proof}

Let us go back to the proof of Proposition \ref{lemma:gysin-geo-trans-bis}.
Using Lemma 
\ref{lem.basechangepitchfork-oriented} we see easily that the following diagram is commutative:
\[ \xymatrix{
	\K^{b/2}_* (M,\partial M) \ar[d]_{\iota^!_{\pitchfork}} \ar[r]^{f_*} & 
	\K_*^{b/2} (D\nu,S\nu) \ar[d]^{j^!_{\pitchfork}} \\
	\K^{b/2}_{*-r} (N) \ar[r]_{(f|_N)_*} & 
	\K^{b/2}_{*-r} (B)
} \]
Indeed, we can break the above diagram into two diagrams:
\begin{equation}\label{2-diagrams}
 \xymatrix{
	\K^{b/2}_* (M,\partial M) \ar[d]_{i^!_{b/2}} \ar[r]^{f_*} & 
	\K_*^{b/2} (D\nu,S\nu)  \ar[d]^{=} \\
	\K^{b/2}_* (D\nu' ,S\nu') \ar[d]_{(j_N)^!_{\pitchfork}} \ar[r]^{(f|_{D\nu'})_*} & 
	\K_*^{b/2} (D\nu,S\nu) \ar[d]^{j^!_{\pitchfork}} \\
	\K^{b/2}_{*-r} (N) \ar[r]_{(f|_N)_*} & 
	\K^{b/2}_{*-r} (B)
}
\end{equation}
The upper square commutes by naturality, whereas the bottom diagram commutes by 
Lemma 
\ref{lem.basechangepitchfork-oriented}. Similarly, using Lemma \ref{lem.basechangeThom-oriented} we have $j^!_{b/2} \circ f_*= (f|_N)_* \circ \iota^!_{b/2}$.
Consider $[M, \Delta_{\cplx}(\nu_M), x,{\rm id}_M]\in \K^{b/2}_* (M,\partial M)$. We claim, and by the above arguments this will suffice for completing the proof of Proposition  \ref{lemma:gysin-geo-trans-bis},  that $$\iota^!_{b/2} [M, \Delta_{\cplx}(\nu_M), x,{\rm id}_M]=
\iota^!_{\pitchfork} [M, \Delta_{\cplx}(\nu_M), x,{\rm id}_M]\,.$$ To see this equality we observe that 
$i_{b/2}^! [M, \Delta_{\cplx}(\nu_M), x,{\rm id}_M]
= [D\nu', \Delta_{\cplx}(\nu_{D\nu'}), i^* x, {\rm id}_{D\nu'}]$ and so we are left with the task of proving that 
$$(j_{N})_{b/2}^!   [D\nu', \Delta_{\cplx}(\nu_{D\nu'}), i^* x, {\rm id}_{D\nu'}]= (j_{N})^!_{\pitchfork}  [D\nu', \Delta_{\cplx}(\nu_{D\nu'}), i^* x, {\rm id}_{D\nu'}].$$
Since $N$ is a closed manifold, we can describe $(j_N)^!_{b/2}$
as in \eqref{explicit-board} where, for the sake of simplicity, we do not write the Sullivan class into the notation of an element
in $\K^{b/2}_*$. Using this description, and the definition of $\cap'$, we have
\begin{align*}
(j_N)^!_{b/2}  [D\nu', i^* x, {\rm id}]
&= (j_N)^!_{b/2}  (i^* x\cap' [(D\nu', S\nu'),1,{\rm id}]) \\
&= (j^*_N)( i^* x) \cap' [(N,1, {\rm id}] \\
&= [N,  (j^*_N)(i^* x), \id] \\
&= (j_N)^!_{\pitchfork} [D\nu' ,i^* x, \id].
\end{align*}
Regarding the last equality, we observe that the identity map
$D\nu' \to D\nu'$ is already transverse to the zero section $N$
and the preimage is precisely $N$.
\end{proof}

\medskip
We can now return to establishing the commutativity of the right-most  square in \eqref{commutativity-immersions-main}.
Using the equality \eqref{pitch'} we just need to show that 
$$j^!_{\an} (\kappa \alpha)= \kappa ( j^!_{\pitchfork} \alpha)\quad\text{in}\quad \K_* (B)[\smlhf] $$
with $j:B\to D\nu$ the inclusion of $B$ into $D\nu$ as the zero section.
Let 
\begin{equation}\label{genericelement} [(M,\partial M), \Delta_{\cplx} (\nu_M),x, f:(M,\partial M)\to (D\nu, S\nu)] 
  \in \K^{b/2}_i (D\nu, S\nu) 
  \end{equation}
be an arbitrary element. Consider $x\in  \K^*(M)[\smlhf] $. We start with the case 
$x\in \K^0 (M)[\smlhf] $; we briefly denote the element appearing in \eqref{genericelement} as
$[M,x,f]$. We have, by the definition of $\kappa$,
\[
j^{!}_\an \kappa  [M,x,f]
= j^{!}_\an f_* \,\sign_K (M,x).
\]
We now consider $N:= f^{-1}(B)$, $g:= f|_N$ and $\nu'=g^* \nu$. Recall $\iota: N\hookrightarrow M$;
by base change, we have \[
j^!_\an  f_* \sign_K (M,x)
 = g_* \iota^!_{\an} \sign_K (M,x)\quad \text{in}\quad \K_* (B)[\smlhf] \]
Now, we know that
\[
 \sign_K (M,x)
  =   [[x]] \otimes \sign_K (M)
\]
and according to Lemma \ref{lem.gysinandcap-bis},
\[
  \iota^!_{\an} ([[x]] \otimes \sign_K (M))
 =  [[\iota^* (x)]] \otimes \iota^!_{\an} \sign_K (M).
\]
By the main result of Part 1 we have 
 $\iota^!_{\an} \sign_K (M)=   \sign_K (N)$.
  Thus we can write:

\[ 
g_* \iota^!_{\an}  \sign_K (M,x)=g_* ([[\iota^* x]] \otimes \iota^!_{\an}  \sign_K (M))
= g_* ([[\iota^* x]] \otimes  \sign_K (N)).
\]
On the other hand, 
\[
 g_* ([[\iota^* x]] \otimes   \sign_K (N))
  = g_*  \sign_K (N, \iota^* x).
\]
Thus, summarizing,
\begin{align*}
  j^{!}_\an \kappa [M,x,f]&=j^{!}_\an f_* \sign_K (M,x)=g_*  \sign_K (N,\iota^* x)\\
  &= \kappa [N, \iota^* x, g] 
  = \kappa j^!_{\pitchfork} [M,x,f]
\end{align*}
This part of the proof, that is when  $x\in \K^0 (M)[\smlhf] $,  is now complete since we have proved that $$j^{!}_\an \kappa [M,x,f]=\kappa j^!_{\pitchfork} [M,x,f]\,.$$
Next we consider  $x\in \K^1 (M)[\smlhf] $. We again denote briefly an element in  
$\K^{b/2}_* (D\nu, S\nu)$ as $ [M,x,f]$.
Recall that, by definition,
$$\kappa  [M,x,f]= (f \circ\pi_1)_* \sign_K (S^1\times M,\sigma_! (x), f\circ \pi_1) $$
with $\pi_1 : M\times S^1 \to M$ the projection onto the first factor and $\sigma:M\to M\times S^1$ the map 
sending $m$ to $(m, (1,0))$. Let
$$\widetilde{N}:=(f\circ \pi_1)^{-1} (B)\,; \;\;\;\widetilde{g}:= (f\circ \pi_1)|_{\widetilde{N}}: 
\widetilde{N}\to B\,;\;\;\; \widetilde{\iota}:\widetilde{N}\hookrightarrow M\times S^1\,.$$
Then, with justifications analogous to the ones we have just employed, we get
\begin{align*}
 j^{!}_\an \kappa [M,x,f]&= j^{!}_\an (f \circ\pi_1)_* \sign_K (S^1\times M,\sigma_! (x))\\
 &= \widetilde{g}_* \widetilde{\iota}^! \sign_K (S^1\times M,\sigma_! (x))\\
 &= \widetilde{g}_* \widetilde{\iota}^! ( [[\sigma_! (x)]]\otimes  \sign_K (S^1\times M))\\
 &= \widetilde{g}_* ([[\widetilde{\iota}^* \sigma_! (x)]]\otimes  \widetilde{\iota}^! (\sign_K (S^1\times M)))\\
 &= \widetilde{g}_* ([[\widetilde{\iota}^* \sigma_! (x)]]\otimes  \sign_K (\widetilde{N}))\\
 &=\kappa [\widetilde{N},\widetilde{\iota}^* \sigma_! (x),\widetilde{g}]\\
 &=\kappa j^!_{\pitchfork}[M\times S^1, \sigma_! (x),f \circ\pi_1]\\
 &=\kappa j^!_{\pitchfork}[M,x,f]
 \end{align*}
 where for the last equality we have used the proof of Proposition \ref{prop-fork-welldefined}.
 
This establishes the commutativity of the right-most square in \eqref{commutativity-immersions-main} and hence finishes the proof of Theorem \ref{thm.compatibility-immersions}.

\section{Compatibility of Gysin maps for K-oriented inclusions and the $\Spin^c$ Dirac operator}\label{sec:CompKoriented}

In section \ref{sec:CompHoriented} we have established the compatibility of the 
Gysin maps induced by an H-orientation. For analytic K-homology (tensored with 
$\bb Z[\tfrac12]$) the corresponding orientations, and hence Gysin maps, were 
induced by the signature operator. In this section we establish the compatibility of 
the Gysin maps induced by a K-orientation. Thus, for analytic K-homology we will 
require $\Spin^c$ structures and then the corresponding orientations and Gysin 
maps are induced by the $\Spin^c$ Dirac operator. (For normally nonsingular 
inclusions this requirement will be for the normal bundle of the inclusion.) As 
the arguments will not differ substantially from those in the previous section, 
we will be brief. 

\subsection{Topological and geometric Gysin maps for K-oriented inclusions}
Let $B$ and $Y$ be topological spaces. Let $g:B\to Y$ be a topologically normally nonsingular inclusion as in Definition \ref{def:nnsi} and assume that the corresponding normal bundle is K-oriented. Let $i: U\subset Y$ be an open tubular neighborhood of $B$ in $Y$
and let $\phi: U\to X$ be a homeomorphism to the total 
space $X$ of the $\Spin^c$ normal bundle $\nu$ with projection $p:X\to B$. 
Let $r$ be the rank of $\nu$ and let us denote by $j$ the inclusion of $B$ into $D\nu$ as the zero-section. 
The Gysin map $g^!_{\topo,\Spin^c}\equiv g^!_{\topo}$ in topological K-homology is obtained as the composition
\[ \K^\topo_n (Y) \to \K^\topo_n (U,U-B) \xlra{\cong}  \K^\topo_n ({\rm Th} \nu, \infty)
\xlra{u_K (\nu) \cap -} \K^\topo_{n-r} (B). \]
where $\K^\topo_n (U,U-B) \xlra{\cong} \K^\topo_n ({\rm Th} \nu, \infty)$ is induced by the homeomorphism $\phi$. The definition of $g^!_{b,\Spin^c}\equiv g^!_b$ is similar but we use $
\K^b_n ({\rm Th} \nu, \infty)
  \xlra{u_K (\nu) \cap' -} \K^b_{n-r} (B)$ in the last step.

\begin{prop} \label{prop.gysincomptopgeovbsoverspinc}
Let $B$ be a finite CW-complex and let $\nu$ be a $\K$-oriented vector
bundle of rank $r$ over $B$. Let $V$ be the total space of $\nu$ and
$j: B \hookrightarrow V$ the zero section. Then the diagram
\[ \xymatrix{
\K^\topo_i (D\nu,S\nu) \ar[d]_{j^!_{\topo,\Spin^c}} 
 & \K^{b}_i (D\nu,S\nu) \ar[l]_-{\varphi}^-\simeq 
    \ar[d]^{ j^!_{b,\Spin^c}} \\
\K^\topo_{i-r} (B) & \K^{b}_{i-r} (B) \ar[l]_{\varphi}^\simeq
} \]
commutes.
\end{prop}
\begin{proof}
The topological Gysin restriction $j^!_{\topo,\Spin^c}$ is by definition given by
multiplication with the Thom class $u_K(\nu) \in \K^r (D\nu,S\nu)$, given that $\nu$ is K-oriented. 
Similarly, $j^!_{\geo,\Spin^c}$ is by definition given by $\cap'$
multiplication with the same Thom class $u_K(\nu) \in \K^r (D\nu,S\nu)$. The statement then follows from
\eqref{jakob-compatibility}.
\end{proof}

\subsection{Gysin maps through transversality.}

Let $B$ be a finite CW-complex
and let $\nu$ be a $\K$-oriented vector
bundle of rank $r$ over $B$. Let $V$ be the total space of $\nu$ and
$j: B \hookrightarrow V$ the zero section. Consider an arbitrary element 
\[ [(M,\partial M), u_{ABS} (\nu_M), x, f:(M,\partial M)\to (D\nu, S\nu)] 
  \in \K^b_i (D\nu, S\nu) \]
  with $M$ a $\Spin^c$ manifold and $x\in \K^* (M)$.
  As for the oriented case already treated, by Thom's transversality theorem
(see Browder \cite[p. 34, Thm. II.2.1]{browder}),
$f$ is homotopic to a map
$\widetilde{f}: (M,\partial M)\to (D\nu, S\nu)$
such that $\widetilde{f}$ is transverse to the zero-section
$B$ in $D\nu$.
This means that $N^{n-r} := \widetilde{f}^{-1} (B)$
is a smooth submanifold of $M^\circ$ with normal vector bundle
$\nu_N$ such that $\widetilde{f}$ restricted to a tubular
neighborhood of $N$ in $M^\circ$ is a linear bundle map
$\nu_N \to \nu$. This defines in particular a vector bundle isomorphism
\begin{equation} \label{equ.nunisognu}
\nu_N \cong g^* \nu, 
\end{equation} 
where $g:= \widetilde{f}|: N\to B$.
We will write $V_N$ for the total space of $\nu_N$.
The smooth manifold $N$ is compact and has empty boundary.
It is oriented, since both $M$ and $\nu_N$ are oriented. Moreover, because of the 2-out-of-3 Lemma
(applied to $M$, $\nu$ and $N$) we have that $N$ is $\Spin^c$. We denote by iota, $\iota:N\hookrightarrow M$, the inclusion.

\begin{defn}
Consider $[(M,\partial M), u_{ABS} (\nu_M), x, f:(M,\partial M)\to (D\nu, S\nu)] 
  \in \K^b_i (D\nu, S\nu) $. Assume without loss of generality that $f$ is transverse 
  to the zero-section $B$ in $D\nu$. Let  $N := f^{-1} (B)$ as above.
  We set $$j^!_{b,\pitchfork}  [(M,\partial M), u_{ABS} (\nu_M), x, f:(M,\partial M)\to (D\nu, S\nu)] := 
[N,u_{ABS} (\nu_M),\iota^* (x), g:N\to B].$$
  \end{defn}
  
The behavior of this map is sufficiently similar to the H-oriented case that we omit the proofs of the following results.

\begin{lemma}\label{lemma:pitchfork} 
  $j^!_{b,\pitchfork}$ induces a well defined homomorphism $\K^b_i (D\nu, S\nu)\to 
  \K^b_{i-r} (B)$.
\end{lemma}
 
\begin{lemma}\label{lemma:gysin-geo-trans}
Let $B$ a finite CW-complex and let $\nu\to B$ a $\Spin^c$ vector bundle over $B$ of rank $r$.
Let $j:B\hookrightarrow \nu$ the inclusion of $B$ as the zero-section of $\nu$.
Consider the Gysin homomorphism
 $ j^!_{b,\Spin^c}: \K^b_i (D\nu, S\nu)\to  \K^b_{i-r} (B)$
 defined  by taking the $\cap'$-product with the Thom class $u_K (\nu)\in \K^r (D\nu,S\nu)$. Then 
\begin{equation}\label{gysin-geo-trans}
j^!_{b,\Spin^c} = j^!_{b,\pitchfork}
\end{equation}
\end{lemma}

\subsection{Compatibility between topological and analytic Gysin maps for K-oriented inclusions}

Let 
\[ [(M,\partial M), u_{ABS}(\nu_M),x, f:(M,\partial M)\to (D\nu, S\nu)] 
  \in \K^b_i (D\nu, S\nu) \]
be an arbitrary element. We shall also employ the shorter notation
$[M,x,f]$, as in Jakob \cite{Jak:BDH}. Recall that by the very definition of $\kappa^b$ we have 
\[
\kappa^b [M,x,f]=
 f_* [D^{\Spin^c}_{M, x}]
\]
if $x\in \K^0 (M)$ and otherwise 
$$\kappa^b ([M, x, f:M\to X]):= 
 (f\circ \pi_1)_* [D^{\Spin^c}_{S(V\oplus 1),\sigma_! (x)}]$$
 if $x\in \K^1 (M)$; here $V\to M$ equals the trivial real line bundle, $\sigma$ is the natural map
 $M\to S(V\oplus 1)=M\times S^1$ and $\pi_1$ is the projection onto the first factor. See the definition given before Proposition \ref{prop:kappawell}. We then have:

\begin{prop} \label{prop.gysincompgeoanavbsoverspinc}
Let $B$ be a finite CW-complex and let $\nu$ be a $\K$-oriented vector
bundle of rank $r$ over $B$. Let $V$ be the total space of $\nu$ and
$j: B \hookrightarrow V$ the zero section inclusion. Then the diagram

\[ \xymatrix{
\K^b_i (D\nu,S\nu) \ar[d]_{j^!_{b,\Spin^c}} \ar[r]^{\kappa}
 & \K^\an_i (D\nu,S\nu) 
    \ar[d]^{j^{!}_{\an,\Spin^c}} \\
\K^b_{i-r} (B) \ar[r]_{\kappa} 
& \K^\an_{i-r} (B)
} \]
commutes.
\end{prop}

\begin{proof}
The Proposition is proved with arguments very similar
 to the ones given in the oriented case, see the last part of the proof of Theorem \ref{thm.compatibility-immersions}.
 Crucial in the proof is the result established in Lemma
 \ref{lemma:gysin-geo-trans}, namely that 
\begin{equation*}
j^!_{b,\Spin^c} = j^!_{b,\pitchfork}
\end{equation*}
\end{proof}

In terms of the composition
\begin{equation*}
	\lambda_{c} = \kappa\circ \phi^{-1}: \K^{\topo}_*(\cdot) \lra \K^{\an}_*(\cdot),
\end{equation*}
the main result of this section is the following.

\begin{thm}\label{thm:big-commutative}
Let $B,Y$ be compact topological spaces having the homotopy type of a finite CW-complex.
Let
$g: B \hookrightarrow Y$ be a $\K$-oriented normally nonsingular inclusion.
Then the diagram
\[ \xymatrix{
	\K^\topo_* (Y) \ar[d]_{g^!_\topo} \ar[r]^-{\lambda_c} &
	\K^\an_* (Y) \ar[d]^{g^{!}_{\an,\Spin^c}} \\
	\K^\topo_* (B) \ar[r]^-{\lambda_c} &
	\K^\an_* (B)
} \]
commutes.
\end{thm}

\begin{proof}
The proof is parallel to the proof of Theorem \ref{thm.compatibility-immersions} so that, thanks to the naturality of $\phi$ and $\kappa,$ we only need to check the commutativity of 
\[ \xymatrix{
\K^b_n (D\nu,S\nu) \ar[d]_{j^!_b} \ar[r]^\varphi
 & \K^\topo_n (D\nu,S\nu) \ar[d]^{j^!_\topo} \\
\K^b_{n-r} (B) \ar[r]^\varphi
 & \K^\topo_{n-r} (B)
} \]
which follows from Proposition \ref{prop.gysincomptopgeovbsoverspinc}, and of 
\begin{equation*}
 \xymatrix{
\K^b_n (D\nu,S\nu) \ar[d]_{j^!_\geo} \ar[r]^\kappa
 & \KK_n (C_0 (X),\cplx) \ar[d]^{\Sigma(p)^{-1} \otimes -} \\
\K^b_{n-r} (B) \ar[r]^\kappa
 & \KK_{n-r} (C(B),\cplx)
} \end{equation*}
which follows from Proposition \ref{prop.gysincompgeoanavbsoverspinc}. Together with the naturality of $\phi$ and $\kappa,$ this finishes the proof of the theorem.

\end{proof}

\section{The submersive case}\label{sec:submersive}

In this section we compare the analytic and topological Gysin maps associated to submersions. In Part 1 we have defined the analytic Gysin map for fiber bundles of Witt spaces but it is clear that the construction only requires that the fibers be Witt spaces and works equally well for fiber bundles of finite CW-complexes as long as the fibers are Witt spaces. (The reason that we assumed in Part I that the base was also a Witt space was so that we could show that the Gysin map sent an orientation of the base to an orientation of the total space.) On the topological side, we will describe Gysin maps for fiber bundles associated to principal bundles with compact Lie structure group. The main result of this section is then the compatibility of the two definitions in this case (so one could think of the Gysin map defined in Part 1 as extending the topologically defined map beyond compact Lie structure groups).

\subsection{The topological Gysin maps in the submersive case}
Let $\xi$ be a normally nonsingular $F$-fiber bundle\footnote{Recall \cite[\S5.4.2]{GorMac:IHTII} that an oriented topological fiber bundle $p:X \lra B$ is normally nonsingular (with codimension $(-c)$) if the fiber $p^{-1}(b)$ is a topological manifold of dimensions $c.$} over a finite CW complex $B$ of dimension $b$
with projection $p:X\to B$.  We assume that $F$ is is a smooth compact d-dimensional manifold and that $\xi$ is associated  to a G-principal bundle with $G$ a compact Lie group.
Thus there exists a principal $G$-bundle $\widetilde{X}\to B$ and an action of $G$ on $F$ by 
diffeomorphism; then $X:= \widetilde{X}\times_G F$. 

\medskip
\noindent
We have defined in Part 1
an analytic bundle transfer
\[p^!_\an: \K^\an_* (B) [\smlhf] \longrightarrow  \K^\an_{*+d} (X)[\smlhf]  \]
We shall recall below, following Banagl (cf. \cite{Ban:BTLOCSS, Ban:TSSKSS}), how to construct  a topological bundle transfer
\[ p^!_\topo: \KO^\topo_* (B) [\smlhf] \longrightarrow \KO^\topo_{*+d} (X)[\smlhf]\,.  \]

\medskip
\noindent
First, we construct a fiberwise embedding
\[ \theta: X \hookrightarrow B \times \real^s,~ s \text{ large},  \]
over $B$ which is normally nonsingular and possesses a normal bundle
admitting the interpretation of a vertical normal bundle for $\xi$.
We use techniques of Becker and Gottlieb \cite{beckergottlieb}.
By the Mostow-Palais equivariant embedding theorem
(\cite{mostow}, \cite{palais})
there exists, for sufficiently large $n$, 
a $G$-module structure on the real vector space $\real^n$ with a
$G$-invariant metric and a smooth equivariant embedding $F \subset \real^n$.
We endow $F$ with the induced metric.
Let $\eta$ be the vector bundle over $B$ with projection
\[ \pi: E(\eta)= \widetilde{X} \times_G \real^n \longrightarrow B. \]
The equivariant embedding $F\subset \real^n$ induces a fiberwise embedding
\[ X = \widetilde{X} \times_G F \subset \widetilde{X} \times_G \real^n = E(\eta) \]
over $B$.
Since $B$ is a finite CW complex, we can choose a complementary vector
bundle $\zeta$ over $B$ with
\[ \eta \oplus \zeta = B \times \real^s,  \]
the trivial rank $s$ vector bundle over $B$.
We obtain an embedding $\theta: X \hookrightarrow B \times \real^s$
as the composition
\[ X \subset E(\eta) \subset E(\eta \oplus \zeta) = B \times \real^s. \]
By construction, the diagram
\[ \xymatrix{
X \ar[dr]_p \ar@{^{(}->}[rr]^-\theta & & B \times \real^s \ar[dl]^{\pro_1} \\
& B &
} \]
commutes.

\begin{lemma}
The fiberwise embedding $\theta$ is a normally nonsingular inclusion of codimension $r:=s-d$.
\end{lemma}
\begin{proof}
Let $TF$ denote the tangent bundle of $F$. 
We will write $TF$ also for its total space.
The group $G$ acts on $TF$ by the derivatives of elements in $G$.
Then the projection $TF \to F$ is $G$-equivariant.
The vertical tangent bundle $TX/B$ of $p: X\to B$  is the vector bundle
over $X$ given by
\[ TX/B = \widetilde{X} \times_G TF \longrightarrow
         \widetilde{X} \times_G F = X. \]
Let $\nu_F$ denote the normal bundle of $F \subset \real^n$.
Let $E(\nu_F) \hookrightarrow \real^n$ be an equivariant embedding
of the total space as a tubular neighborhood of $F$ in $\real^n$.
Let $U \subset \real^n$ be the open image of this embedding.
Let 
\[ TF \oplus \nu_F \cong F \times \real^n  \]
be the trivialization associated with the embedding.
The normal bundle $\nu_F$ of $F$ gives rise to a vector bundle
$\mu$ over $X$ with projection
\[ \pi_\mu: E(\mu) = \widetilde{X} \times_G E(\nu_F) 
    \longrightarrow \widetilde{X} \times_G F  \]
induced by the projection $E(\nu_F) \to F$ of $\nu_F$.
This bundle admits the interpretation as the normal bundle
of $X$ in $E(\eta)$.
The trivialization $TF \oplus \nu_F \cong F \times \real^n$
induces an equivalence
\[ TX/B \oplus \mu \cong p^* (\eta).  \]
We then obtain the trivialization
\[ TX/B X \oplus \mu \oplus p^* (\zeta) \cong 
    p^* (\eta) \oplus p^* (\zeta) = X \times \real^s. \]
The vector bundle $\nu_p$ over $X$ given by
\[ \nu_p := \mu \oplus p^* (\zeta) \]
and thus satisfying
\[ TX/B \oplus \nu_p = X \times \real^s, \]
stably has the interpretation of the \emph{stable vertical normal bundle} of
$p:X\to B$.
At the same time, $\nu_p$ is stably also the (ordinary) normal bundle
of the embedding $\theta: X \hookrightarrow B\times \real^s$, as we
will now show.
Recall that $U\subset \real^n$ is an open tubular $G$-neighborhood of $F$
with an equivariant diffeomorphism $\psi_F: E(\nu_F) \to U$
which restricts to the identity on $F$.
The equivariant open inclusion $U\subset \real^n$ induces
an open inclusion
\[ N := \widetilde{X} \times_G U \subset 
   \widetilde{X} \times_G \real^n = E(\eta).  \]
The projection $N\to B$ induced by $\widetilde{p}: \widetilde{X} \to B$
makes $N$ a $U$-fiber bundle over $B$.
The equivariant inclusion $F \subset U$ induces an inclusion
\[ X = \widetilde{X} \times_G F 
    \subset \widetilde{X} \times_G U = N. \]
So $N$ is an open neighborhood of $X$ in $E(\eta)$.
We claim that $N$ is in fact a tubular neighborhood of $X$ in $E(\eta)$.
Indeed, $\psi_F$ induces a homeomorphism $\psi$ such that
\[ \xymatrix{
E(\mu) = \widetilde{X} \times_G E(\nu_F) 
 \ar[rr]^\psi_\cong & & \widetilde{X} \times_G U = N \\
& X = \widetilde{X} \times_G F \ar@{^{(}->}[ul]^0 \ar@{^{(}->}[ur] & 
} \]
commutes. 
We note that $\psi$ is a map over $B$, i.e. the diagram
\[ \xymatrix{
E(\mu) \ar[d]_\psi^\cong \ar[r]^{\pi_\mu} & X \ar[d]^p \\
N \ar[r] & B
} \]
commutes, since the commutative diagram to a point,
\[ \xymatrix{
E(\nu_F) \ar[d]_{\psi_F}^\cong \ar[r] & F \ar[d] \\
U \ar[r] & \pt,
} \]
induces upon application of $1_{\widetilde{X}} \times_G -$ a commutative
diagram
\[ \xymatrix{
E(\mu)=\widetilde{X} \times_G E(\nu_F) \ar[d]_{\psi}^\cong \ar[r]^-{\pi_\mu} 
  & \widetilde{X} \times_G F =X \ar[d]^p \\
N=\widetilde{X} \times_G U \ar[r] & \widetilde{X} \times_G \pt =B.
} \]
The total space
\[  T := E((\pi^* \zeta)|_N),  \]
with $\pi$ the projection $E(\eta)\to B$ of $\eta,$
is open in $E(\eta \oplus \zeta) = B\times \real^s$. It contains $X$ and thus is
a neighborhood of it in $B\times \real^s$.
We shall verify that $T$ can be given the structure of a vector bundle over $X$,
thus proving that $T$ is tubular. That vector bundle will be seen to be $\nu_p$.
If $\alpha$ and $\beta$ are any two vector bundles over a space $X$, 
then $E(\alpha \oplus \beta) = E(q^* \beta),$ where $q: E(\alpha)\to X$
is the projection of $\alpha$.
Hence,
\[ E(\nu_p) = E(\mu \oplus p^* \zeta) = E(\pi_\mu^* p^* \zeta). \]
Using the commutative diagram
\[ \xymatrix{
E(\mu) \ar[d]_\psi^\cong \ar[r]^{\pi_\mu} & X \ar[d]^p \\
N \ar@{^{(}->}[d] \ar[r] & B \\
E(\eta), \ar[ru]_\pi &
} \]
we may write 
\[ E(\pi_\mu^* p^* \zeta) = E(\psi^* (\pi^* \zeta)|_N),  \]
so that
\[ E(\nu_p) = E(\psi^* (\pi^* \zeta)|_N). \]
Then the top horizontal arrow of the cartesian diagram
\[ \xymatrix{
E(\nu_p) \ar[r]^-\phi_-\cong \ar[d] & E((\pi^* \zeta)|_N) =T \ar[d] \\
E(\mu) \ar[r]^\cong_\psi & N
} \]
defines a homeomorphism $\phi$ that identifies the total space of
$\nu_p$ with the tubular neighborhood $T$ of $X$ in $B\times \real^s$,
as was to be shown.
\end{proof}

\begin{remark}
It is worthwhile to discuss the special case where $B,$ and thus $X$,
are smooth manifolds. Then $\theta:X \hookrightarrow B\times \real^s$ is
a smooth embedding. If $\nu_{X\subset E(\eta)}$ denotes the normal bundle of
$X$ in $E(\eta)$, then
\begin{equation} \label{equ.txplusnuxineetaisteetaonx}
TX \oplus \nu_{X\subset E(\eta)} = T(E\eta)|_X. 
\end{equation}
Furthermore, if $\nu_{E\eta \subset E(\eta \oplus \zeta)}$ 
denotes the normal bundle of $E\eta$ in $E(\eta \oplus \zeta)$, then
\[ T(E\eta) \oplus \nu_{E\eta \subset E(\eta \oplus \zeta)} 
       = T(B\times \real^s)|_{E\eta}. \]
Restricting this to $X$, we obtain
\[ T(E\eta)|_X \oplus (\nu_{E\eta \subset E(\eta \oplus \zeta)})|_X 
       = T(B\times \real^s)|_X. \]
Substituting (\ref{equ.txplusnuxineetaisteetaonx}), we have
\[ TX \oplus \nu_{X\subset E(\eta)} \oplus 
   (\nu_{E\eta \subset E(\eta \oplus \zeta)})|_X 
       = T(B\times \real^s)|_X. \]
Since
\[ \nu_{X\subset E(\eta)} = \mu,~
     \nu_{E\eta \subset E(\eta \oplus \zeta)} = \pi^* \zeta, \]
where $\pi: E(\eta) \to B$ is the projection of $\eta$, we arrive at
\[ TX \oplus \mu \oplus 
   (\pi^* \zeta)|_X 
       = T(B\times \real^s)|_X. \]
So the normal bundle $\nu_\theta$ of $\theta$ has the description
\[ \nu_\theta = \mu \oplus (\pi^* \zeta)|_X. \]
Now, $\nu_p$ is by definition
\[ \nu_p = \mu \oplus p^* \zeta. \]
Note that since 
\[ \xymatrix{
X \ar[dr]_p \ar@{^{(}->}[rr] & & E(\eta) \ar[dl]^\pi \\
& B &
} \]
commutes, we have $p^* \zeta = (\pi^* \zeta)|_X$.
This shows that
\[ \nu_\theta = \nu_p. \]
\end{remark}

We shall now recall the definition of
\[ p^!_\topo: \KO^\topo_* (B) [\smlhf] \longrightarrow \KO^\topo_{*+d} (X)[\smlhf]  \]
following Banagl.
The homomorphism $p^!_\topo $ will be the composition of 3 homomorphisms. Consider $\nu_\theta$, 
associated to the normally non-singular inclusion of $X$ into $B\times \mathbb{R}^s$. This is an orientable real bundle 
of rank $r$; we consider  the associated Sullivan class 
$\Delta (\nu_\theta)\in \widetilde{\KO}^r({\rm Th} (\nu_\theta))[\smlhf] \equiv\KO^r (D\nu_\theta,S\nu_\theta)[\smlhf] $.
Through the Thom-Pontrjagin collapse  $T(\pro_1): {\rm Th}(B\times\mathbb{R}^s) \to {\rm Th} (\nu_\theta)$ 
we get a homomorphism
 $T(\pro_1)_*: \widetilde{\KO}^\topo_* ({\rm Th}(B\times\mathbb{R}^s))[\smlhf]  \to \widetilde{\KO}^\topo_*({\rm Th} (\nu_\theta))[\smlhf] $. Observe that
${\rm Th}(B\times\mathbb{R}^s) = S^s B^+$, the s-th suspension of $B^+$. Overall, we obtain 
$$T(\pro_1)_*: \widetilde{\KO}^\topo_* (S^s B^+)[\smlhf]  \to \KO^\topo_* (D\nu_\theta,S\nu_\theta)[\smlhf]\,.$$
We also consider the suspension isomorphism:
$${\rm Susp}: \KO^\topo_* (B)[\smlhf] \equiv \widetilde{\KO}^\topo_* (B^+)[\smlhf] \rightarrow  \widetilde{\KO}^\topo_{*+s} (S^s B^+)[\smlhf]$$
and, finally, the  homomorphism
$$\Phi: \KO_{*+s}^\topo (D\nu_\theta,S\nu_\theta)[\smlhf]\xrightarrow{\Delta (\nu_\theta)\cap -}  \KO^\topo_{*+(s-r)} (D\nu_\theta)[\smlhf]\equiv 
\KO^\topo_{*+d} (X)[\smlhf]$$
where we have used that $s-r=d$ and that $D\nu_\theta$ is homotopically equivalent to $B$.
The topological Gysin homomorphism associated to $p:X\to B$ is given, by definition, by the composition
$$\KO^\topo_* (B)[\smlhf]\xrightarrow{{\rm Susp}} \widetilde{\KO}^\topo_{*+s} (S^s B^+)[\smlhf]\xrightarrow{T(\pro_1)_*}
\KO^\topo_{*+s} (D\nu_\theta,S\nu_\theta)[\smlhf]\xrightarrow{\Phi} \KO_{*+d} (X)[\smlhf]$$
that is, 
$$p^!_\topo:= \Phi \circ T(\pro_1)_*\circ {\rm Susp}: \KO^\topo_* (B)[\smlhf]\to  \KO^\topo_{*+d} (X)[\smlhf]\,.$$

\subsection{An alternative description of the analytic Gysin map in the submersive case.}
We now pass to the analytic counterpart of this Gysin homomorphism.
We have defined  in Part 1 the analytic transfer map 
\[ \K^\an_* (B) [\smlhf] \xrightarrow{p^!_\an} \K^\an_{*+d} (X)[\smlhf], \]
 using the family of signature operators along the
fibers of $p:X\to B$.
In the particular case we are considering in this subsection, more in line with the definition  adopted by Banagl in the topological context, we can give a different but  equivalent description of this homomorphism. We explain this now.

We have factored $p$ as a composition of a normally nonsingular embedding
$\theta$, of codimension $r=s-d$, and a trivial vector bundle projection:
\[ \xymatrix{
X \ar@{^{(}->}[r]^-\theta \ar[rd]_p & B \times \real^s \ar[d]^{\pro_1} \\
& B
} \]
Denote by $U_\theta$ the tubular neighbourhood of $X$ in $B\times \real^s$ and let $\phi:U_\theta\to \nu_\theta$
the homeomorphism identifying $U_\theta$ with the normal bundle $\nu_\theta$.
We have a suspension isomorphism 
$$ 
\KK^* (C(B),\mathbb{C})[\smlhf]\longrightarrow 
\KK^{*+s} (C_0 (B\times \mathbb{R}^s),\mathbb{C})[\smlhf]$$ which is in fact given by Kasparov multiplication
on the left by the Kasparov element 
\[  \tau_B (\alpha_s) \in \KK^s (C_0 (B\times \real^s), C(B))[\smlhf]\,.  \]
(See \cite[\S 18.9]{Bla:KOA} for the notation $\tau_B$ and \cite[Remarque 2.18(2)]{Hil:FKBPVL} for the notation $\alpha_s$ used here.)
Notice that this element is nothing but the element
$\Sigma (\pro_1)$ associated to the trivial bundle $B\times\real^s \xrightarrow{\pro_1} B$ and so 
we conclude that the suspension isomorphism in analytic K-homology is simply given by $(\pro_1)^!\equiv 
\Sigma (\pro_1)\otimes -$.
Next we have
$$T(\pro_1)_*: \KK^{*+s} (C_0 (B\times \mathbb{R}^s),\mathbb{C})[\smlhf]  \to 
\KK^{*+s} (C_0 (\nu_\theta),\mathbb{C})[\smlhf] \equiv
\K^\an_{*+s} (D\nu_\theta,S\nu_\theta)[\smlhf]$$
which is in fact the composition of the  restriction homomorphism to the tubular neighbourhood $U_\theta$, induced by the
inclusion $i:U_\theta \hookrightarrow B\times\real^s$ and extension by zero,
 followed by
the isomorphism in K-theory induced by the homeomorphism  $\phi: U_\theta\to \nu_\theta$; viz.
$$T(\pro_1)_* = ([\phi]\otimes i!\otimes -)$$
with $[\phi]\in \KK^0 (C_0 (\nu_\theta), C_0 (U_\theta))$, $i ! \in \KK^0 (C_0 (U_\theta), C_0 (B\times \real^s))$
and $[\phi]\otimes i ! \in \KK^0 (C_0 (\nu_\theta), C_0 (B\times \real^s))$.
Finally, we have our usual Gysin homomorphism
$$\K^\an_{*+s} (D\nu_\theta,S\nu_\theta)[\smlhf]\to \K^\an_{*+d} (X)$$
obtained by Kasparov multiplication by $\Sigma (\pi)^{-1}$, with $\pi$ the bundle projection of $\nu_\theta$ onto $X$.
Here we have used  that $s-r=d$, so that $*+s=*+(s-r)+r=*+d+r$.

\begin{prop} \label{prop.analyticbundletransfer}
Let $p:X\to B$ be a  $F^d$-fiber bundle as above. Then the analytic bundle transfer
\[ p^!_\an: \KK^* (C(B),\cplx)[\smlhf] \longrightarrow \KK^{*+d} (C(X),\cplx)[\smlhf] \]
defined in Part 1 
is defined by the commutativity of
\begin{equation}\label{composition-fiber}
\xymatrix@C=50pt{
	\KK^* (C(B),\cplx)[\smlhf] \ar[d]_-{\tau_B (\alpha_s)\otimes -} \ar[r]^-{p^!_{\an}}&
	\KK^{*+d} (C(X),\cplx)[\smlhf] \\
	\KK^{*+s} (C_0 (B\times \real^s),\cplx)[\smlhf] \ar[r]^-{T(\pro_1)_*} &
	\KK^{*+s} (C_0(\nu_\theta),\cplx)[\smlhf] \ar[u]_-{\Sigma(\pi)^{-1}\otimes -}
}
 \end{equation}
\end{prop}

\begin{proof}
Following our discussion above, we see that we  need to show that $$p^!_\an= (\Sigma(\pi)^{-1}\otimes -)  \circ ([\phi] \otimes i !\otimes - )\circ (\Sigma (\pro_1)\otimes -).$$
On the left we have $(\Sigma (p) \otimes -)$ with $\Sigma(p) \in \KK^d (C(X), C(B))$ the bivariant class
defined by the family of signature operators along the fibers of $p:X\to B$. Thus we need to show
that
\begin{equation}\label{alternative-submersive}
(\Sigma (p) \otimes -) = (\Sigma(\pi)^{-1}\otimes -)  \circ ([\phi] \otimes  i ! \otimes - )\circ (\Sigma (\pro_1)\otimes -)
\end{equation}
as morphisms from 
$\KK^* (C(B),\cplx)[\smlhf] $ to $ \KK^{*+d} (C(X),\cplx)[\smlhf]$.
Observe that $\nu_\theta$ fibers over $B$ through $p\circ \pi$.  Now,
\eqref{alternative-submersive} is true if and only if it is true that 
$$(\Sigma(\pi)\otimes -) \circ (\Sigma (p) \otimes -) 
 = ([\phi] \otimes i ! \otimes )\circ (\Sigma (\pro_1)\otimes -)$$ 
 which can be rewritten, thanks to the functoriality result with respct to 
 compositions proved in Part 1, as  
$$(\Sigma(p\circ \pi)\otimes -) = ([\phi] \otimes i ! \otimes )\circ (\Sigma (\pro_1)\otimes -)\,.$$ 
Observe that
$p\circ \pi$ is equal to $q:=\pro_1\circ i \circ \phi^{-1}$. As $q:\nu_\theta\to B$ is a fibration,  
$$\Sigma(q)\otimes -: \KK_*(C(B),\cplx)\to  \KK_{*+d}(C_0(\nu_\theta),\cplx)$$ is equal to $([\phi] \otimes i ! \otimes )\circ (\Sigma (\pro_1)\otimes -)$. Thus
$$(\Sigma (p\circ\pi)\otimes -)=(\Sigma(q)\otimes -)= ([\phi] \otimes i ! \otimes )\circ (\Sigma (\pro_1)\otimes -)$$ and the proposition is proved.
\end{proof}

\subsection{Compatibility of analytic and topological Gysin maps in the submersive case.}
We are now ready to state and prove the main result of this subsection. Recall, from \eqref{eq:def-of-lambda}, the homomorphism
$$\lambda: \KO^\topo_* ( - )[\smlhf] \longrightarrow \K^\an_* ( - )[\smlhf].$$

\begin{thm}
Let $B$ a finite CW-complex and let $p:X\to B$ be a fiber bundle with smooth compact fiber $F$ 
of dimension $d$ and associated to a principal bundle with compact 
Lie structure group $G$. Then the diagram
\begin{equation}\label{compatibility-fiber-b}
 \xymatrix{
\KO^\topo_* (B)[\smlhf] \ar[d]_{p^!_\topo} \ar[r]^{\lambda} &
  \K^\an_{*}(B)[\smlhf] \ar[d]^{p^!_\an} \\
  \KO^\topo_{*+d} (X)[\smlhf]  \ar[r]_\lambda &  \K^\an_{*+d} (X)[\smlhf]
} 
\end{equation}
commutes.
\end{thm}

\begin{proof}
By Proposition \ref{prop.analyticbundletransfer} we can take as the right vertical homomorphism 
the composition appearing in \eqref{composition-fiber}. We can then break the diagram in \eqref{compatibility-fiber-b}
into the following 3 squares:
\begin{equation*}
 \xymatrix{
	\KO^\topo_* (B)[\smlhf] \ar[d]_{{\rm Susp}} \ar[r]^-{\lambda} &
	\KK^{*}(C(B),\cplx)[\smlhf] \ar[d]^{\tau_B (\alpha^s)\otimes -} \\
	\widetilde{KO}^\topo_{*+s} (S^s B^+)[\smlhf]  \ar[d]_{T(\pro_1)_*}\ar[r]_-\lambda &  
	\KK^{*+s} (C_0 (B\times \real^s),\cplx)[\smlhf] \ar[d]^{T(\pro_1)_*}\\
	\KO^\topo_{*+s} (D\nu_\theta,S\nu_\theta)[\smlhf] \ar[d]_{\Delta (\nu_\theta)\cap -}\ar[r]_-\lambda &  
	\KK^{*+s} (C_0 (\nu_\theta),\cplx)[\smlhf] \ar[d]^{\Sigma(\pi)^{-1}\otimes -}\\
	\KO^\topo_{*+d} (X)[\smlhf]\ar[r]_-\lambda & 
	\KK^{*+d} (C (X),\cplx)[\smlhf]
} 
\end{equation*}
Consider the top square; we have already remarked  that the vertical map on the right hand side is precisely the suspension 
isomorphism in analytic K-homology. 
The upper diagram then commutes since $\lambda$ is a transformation of homology theories.
The commutativity of the central square is clear. The commutativity of the bottom square has been already discussed in
a previous section, see \eqref{commutativity-immersions-main}.
\end{proof}


\begin{thebibliography}{ALMP18}

\bibitem[ABP25]{Part1}
Pierre Albin and Markus Banagl and Paolo Piazza.
\newblock Smooth atlas stratified spaces, K-Homology Orientations, and Gysin maps.
\newblock preprint, available online at {https://arxiv.org/abs/2505.14952}.

\bibitem[ABS64]{AtiBotSha:CM}
M.~F. Atiyah, R.~Bott, and A.~Shapiro.
\newblock Clifford modules.
\newblock {\em Topology}, 3(suppl):3--38, 1964.

\bibitem[Ban24]{Ban:BTLOCSS}
Markus Banagl.
\newblock Bundle Transfer of L-Homology Orientation Classes
  for Singular Spaces.
\newblock Algebraic \& Geometric Topology \textbf{24} (2024), 2579 -- 2618,
  DOI: 10.2140/agt.2024.24.2579

\bibitem[Ban25]{Ban:TSSKSS}
Markus Banagl.
\newblock Transfer and the Spectrum-Level Siegel-Sullivan KO-Orientation for
  Singular Spaces.
\newblock J. of Topology and Analysis \textbf{18} (2026), 371 -- 407, 
  DOI: \texttt{https://doi.org/10.1142/S1793525324500341}

\bibitem[BLM19]{BanLauMcC:LFCISNC}
Markus Banagl, Gerd Laures, and James~E. McClure.
\newblock The {$L$}-homology fundamental class for {IP}-spaces and the
  stratified {N}ovikov conjecture.
\newblock {\em Selecta Math. (N.S.)}, 25(1):Paper No. 7, 104, 2019.

\bibitem[BD82]{BauDou:HIT}
Paul Baum and Ronald~G. Douglas.
\newblock {$\K$}\ homology and index theory.
\newblock In {\em Operator algebras and applications, {P}art 1 ({K}ingston,
  {O}nt., 1980)}, volume~38 of {\em Proc. Sympos. Pure Math.}, pages 117--173.
  Amer. Math. Soc., Providence, RI, 1982.

\bibitem[BHS07]{baumhigsonschick} 
P. Baum, N. Higson, T. Schick,
\newblock On the Equivalence of Geometric and Analytic K-Homology,
\newblock Pure Appl. Math. Q., \textbf{3} (1, part 3) (2007), 1 -- 24.

\bibitem[BG75]{beckergottlieb} J. C. Becker, D. H. Gottlieb,
\newblock The Transfer Map and Fiber Bundles,
 \newblock Topology \textbf{14} (1975), 1 -- 12.

\bibitem[Bla86]{Bla:KOA}
Bruce Blackadar.
\newblock {\em {$\K$}-theory for operator algebras}, volume~5 of {\em
  Mathematical Sciences Research Institute Publications}.
\newblock Springer-Verlag, New York, 1986.

\bibitem[BOA66]{boardman} 
J. M. Boardman,
 \newblock {\em Stable Homotopy Theory},
  mimeographed notes, University of Warwick, 1965-66

\bibitem[BR72]{browder} 
W. Browder,
 \newblock {\em Surgery on Simply-Connected Manifolds},
\newblock Ergebnisse der Mathematik und ihrer Grenzgebiete, 2. Folge, vol. 65,
 Springer-Verlag Berlin Heidelberg, 1972

\bibitem[BRS76]{buonrs}
  S. Buoncristiano, C. P. Rourke, B. J. Sanderson,
  {\em A Geometric Approach to Homology Theory},
  London Math. Soc. Lecture Note Series, vol. 18,
  Cambridge Univ. Press, 1976.

\bibitem[CS81]{ConSka:TLPF}
A.~Connes and G.~Skandalis.
\newblock Th\'{e}or\`eme de l'indice pour les feuilletages.
\newblock {\em C. R. Acad. Sci. Paris S\'{e}r. I Math.}, 292(18):871--876,
  1981.

\bibitem[CS84]{ConSka:LITF}
A.~Connes and G.~Skandalis.
\newblock The longitudinal index theorem for foliations.
\newblock {\em Publ. Res. Inst. Math. Sci.}, 20(6):1139--1183, 1984.

\bibitem[GM80]{GorMac:IHTII}
Mark Goresky and Robert MacPherson.
\newblock Intersection homology theory II.
\newblock {\em Invent. Math.} 72 (1983), no. 1, 77–129.

\bibitem[GM88]{GorMac:SMT}
Mark Goresky and Robert MacPherson.
\newblock Stratified Morse theory
\newblock {\em Ergeb. Math. Grenzgeb.} (3), 14[Results in Mathematics and Related Areas (3)]
Springer-Verlag, Berlin, 1988, xiv+272 pp.

\bibitem[Hil89]{Hil:FKBPVL}
Michel Hilsum.
\newblock Fonctorialit\'{e} en {$\K$}-th\'{e}orie bivariante pour les
  vari\'{e}t\'{e}s lipschitziennes.
\newblock {\em $\K$-Theory}, 3(5):401--440, 1989.

\bibitem[Jak98]{Jak:BDH}
Martin Jakob.
\newblock A bordism-type description of homology.
\newblock {\em Manuscripta Math.}, 96(1):67--80, 1998.

\bibitem[Jak99]{Jak:cont}
Martin Jakob.
\newblock An alternative approach to homology.
\newblock Une d\'egustation topologique [Topological morsels]: homotopy
      theory in the Swiss Alps. Contemp. Math.,
      vol 265, 87--97, 1999
   
\bibitem[MM79]{madsenmilgram} I. Madsen, R. J. Milgram,
 {\em The Classifying Spaces for Surgery and Cobordism of Manifolds},
 Annals of Math. Studies \textbf{92}, Princeton University Press, 1979.

\bibitem[McC75]{mccrory} C. McCrory,
 {\em Cone Complexes and PL Transversality},
 Transactions of the American Math. Soc. \textbf{207} (1975),
 269 -- 291.  

\bibitem[MO57]{mostow} G. D. Mostow, 
\newblock Equivariant Embeddings in Euclidean Space, 
\newblock Annals of Mathematics, Second Series, \textbf{65} (1957), 432 -– 446.          
     
\bibitem[PA57]{palais} R. S. Palais, 
\newblock Imbedding of compact, differentiable transformation groups in 
    orthogonal representations, 
\newblock Journal of Mathematics and Mechanics, \textbf{6} (1957), 673 -– 678.      
 
\bibitem[Ros12]{rosenbergncgexampandapps} J. Rosenberg,
 {\em Examples and Applications of Noncommutative Geometry and K-Theory},
\newblock In {\em Topics in noncommutative geometry}, volume~16 of {\em Clay
  Math. Proc.}, pages 93--129. Amer. Math. Soc., Providence, RI, 2012.

\bibitem[RS82]{rourkesanderson} C. P. Rourke, B. J. Sanderson,
 {\em Introduction to Piecewise-Linear Topology},
 Springer Study Edition, Springer-Verlag, 1982.

\bibitem[Rud98]{Rud:TSOC}
Yuli~B. Rudyak.
\newblock {\em On {T}hom spectra, orientability, and cobordism}.
\newblock Springer Monographs in Mathematics. Springer-Verlag, Berlin, 1998.
\newblock With a foreword by Haynes Miller.

\bibitem[RW06]{RosWei:SO}
Jonathan Rosenberg and Shmuel Weinberger.
\newblock The signature operator at 2.
\newblock {\em Topology}, 45(1):47--63, 2006.

\bibitem[Sie83]{Sie:WSGCTKOP}
P.~H. Siegel.
\newblock Witt spaces: a geometric cycle theory for {$\K{\rm O}$}-homology at
  odd primes.
\newblock {\em Amer. J. Math.}, 105(5):1067--1105, 1983.

\bibitem[Sul71]{Sul:GTP}
Dennis Sullivan.
\newblock {\em Geometric topology. {P}art {I}}.
\newblock Massachusetts Institute of Technology, Cambridge, MA, 1971.
\newblock Localization, periodicity, and Galois symmetry, Revised version.

\end{thebibliography}
\end{document}